\documentclass[11pt,reqno]{amsart}

\usepackage{amsfonts}
\usepackage{amsmath}
\usepackage{amssymb}
\usepackage{amsthm}
\numberwithin{equation}{section}

\usepackage{anysize}
\marginsize{2cm}{2cm}{2cm}{2cm}

\usepackage{xcolor}

\usepackage{titlesec}

\titleformat{\section}
  {\normalfont\normalsize\scshape\color{black}\centering}
  {\thesection}
  {1em}
  {}

\titleformat{\subsection}[runin]
  {\normalfont\normalsize\bfseries}
  {\thesubsection}
  {1em}
  {}
  [.]
  
\usepackage[skip=3pt, indent=10pt]{parskip}

\makeatletter
\renewenvironment{proof}[1][\proofname]{\par
  \pushQED{\qed}%
  \normalfont \topsep6\p@\@plus6\p@\relax
  \trivlist
  \item[\hskip\labelsep
        \bfseries\itshape #1\@addpunct{.}]\ignorespaces
}{%
  \popQED\endtrivlist\@endpefalse
}
\makeatother

\usepackage{xcolor}

\usepackage{hyperref}
\usepackage[numbers,sort,square]{natbib}

\usepackage{mathrsfs}

\usepackage{mathtools}

\usepackage{verbatim}

\newtheorem{corollary}{Corollary}[section]

\newtheorem{lemma}{Lemma}[section]
\newtheorem{theorem}{Theorem}[section]
\newtheorem{proposition}{Proposition}[section]
\newtheorem{remark}{Remark}[section]

\newcommand{\NN}{{\mathbb N}}
\newcommand{\TT}{{\mathbb T}}
\newcommand{\ZZ}{{\mathbb Z}}
\newcommand{\RR}{{\mathbb R}}

\title[The Bramson correction on homogeneous trees]
{Logarithmic Bramson correction for Fisher--KPP equations on homogeneous trees}

\date{\today}

\author[E. Alvarez]{Edgardo Alvarez\textsuperscript{1}}
\address{Universidad del Norte, Departamento de Matem\'aticas y Estad\'istica, Barranquilla, Colombia.}
{\email{ealvareze@uninorte.edu.co}}

\author[M. Murillo]{Marina Murillo-Arcila\textsuperscript{2}}
\address{Departamento de Matem\'aticas, Facultad de Ciencias, Universidad de C\'adiz E-11519, Puerto Real, Spain.}
\thanks{\textsuperscript{2}Partially funded by PID2022-139449NB-I00, MCIN/AEI/10.13039/501100011033/FEDER,
UE}
{\email{marina.murillo@uca.es}}

\author[R. Ponce]{Rodrigo Ponce\textsuperscript{3}}
\address{Instituto de Matem\'aticas, Universidad de Talca, Casilla 747, Talca, Chile. }
\thanks{\textsuperscript{3}Partially funded by Chilean research grant FONDECYT 1251151}
{\email{rponce@inst-mat.utalca.cl}}

\author[J.C. Pozo]{Juan C. Pozo\textsuperscript{4}}
\address{Instituto de Ciencias de la Ingeniería, Universidad de O'Higgins, Rancagua, Chile.}
\thanks{\textsuperscript{4}Partially funded by Chilean research grant FONDECYT 1251151}
{\email{juan.pozo@uoh.cl}}

\begin{document}

\begin{abstract}
In this paper, we study the Fisher--KPP equation
\[
\partial_t u
=
\alpha\Delta_{\TT_{q+1}}u+\beta u(1-u),
\]
on the homogeneous \((q+1)\)-regular tree \(\TT_{q+1}\), with \(q\geq2\),
for nontrivial, compactly supported radial initial data. We first obtain an
explicit representation of the fundamental solution of the corresponding
linearized problem, together with sharp two-sided pointwise estimates. These
estimates identify the propagation threshold
\[
\beta>\alpha(\sqrt q-1)^2
\]
and determine the associated critical speed \(c_\ast\) and decay rate
\(\lambda_\ast\). In this regime, we prove that, for every fixed
\(\theta\in(0,1)\), the outermost \(\theta\)-level set satisfies
\[
\kappa_\theta(t)
=
c_\ast t-\frac{3}{2\lambda_\ast}\log t+O(1),
\qquad t\to\infty.
\]
Therefore, while the branching geometry affects the threshold for propagation
and the values of the critical parameters, the classical Bramson factor
\(3/2\) persists. Our results therefore extend classical propagation and
logarithmic-delay phenomena, previously known for one-dimensional continuous
and discrete Fisher--KPP models, to homogeneous trees.
\end{abstract}


\keywords{Fisher--KPP equation, homogeneous trees, Bramson correction,
traveling fronts, discrete heat kernel, reaction--diffusion on graphs}

\subjclass[2020]{Primary 35K57; Secondary 35B40, 35C07}
\maketitle


\section{Introduction}

\subsection*{Motivation} The Fisher--KPP equation is one of the basic models for propagation into an
unstable state.  In its classical one-dimensional form,
\[
\partial_tu=\alpha\partial_{xx}u+f(u),
\]
a Fisher--KPP nonlinearity satisfies \(f(0)=f(1)=0\), \(f'(0)>0\), and
\(0<f(s)\leq f'(0)s\) for \(s\in(0,1)\).  It is well known, see, e.g., \cite{Aronson-Weinberger-1978-Advances} that for sufficiently localized nontrivial initial data, the invasion speed, denoted by $c_\ast$, is obtained from the
linearization at \(u=0\).  Particularly, in the logistic case
\(f(u)=\beta u(1-u)\), such a speed is \(c_\ast=2\sqrt{\alpha\beta}\).

The spreading speed \(c_\ast\) provides the leading-order description of
the invasion front, but it does not determine its position with complete
accuracy. A finer asymptotic description reveals a universal second-order
effect. In the continuous Fisher--KPP equation, Bramson \cite{Bramson-1978} showed that the
front is delayed logarithmically with respect to the trajectory \(c_\ast t\).
More precisely, in the usual notation, the position of a fixed level set
has the form
\[
\kappa_\theta(t)=c_\ast t-\frac{3}{2\lambda_\ast}\log t+O(1),
\qquad t\to\infty,
\]
where \(\lambda_\ast\) is the critical exponential decay rate. Hence, the first-order law \(c_\ast t\) captures the
propagation speed, whereas the Bramson correction identifies the front
location up to a bounded error. The factor \(3/2\) is a signature of
pulled-front selection and has subsequently appeared in several local and
nonlocal settings, see, for instance, \cite{Bouin-Henderson-Ryzhik-2020-Poicare,Alfaro-Giletti-Xiao-2025-MathAnn}.

Fisher--KPP equations and, more generally, monostable reaction--diffusion
equations in discrete media have been extensively studied over the last
decades, particularly with regard to the existence and stability of
traveling fronts, spreading speeds, and propagation phenomena on lattices
and graphs; see, for instance, \cite{Zinner-Harris-Hudson-1993-JDE, Chen-Guo-2003-MathAnn, Hoffman-Holzer-2019,Guo-Wu-2008-Osaka,Fang-Li-Lou-Wang-2026-JDE} and the references therein.

In such a context, the discrete version of Bramson's logarithmic correction for the Fisher--KPP
equation on \(\ZZ\) was recently established by Besse, Faye, Roquejoffre,
and Zhang \cite{Besse-Faye-Roquejoffre-Zhang-2023-TAMS}. Their analysis
makes clear that the logarithmic correction is not a consequence of a
formal continuum approximation: it is encoded in refined asymptotics of
the discrete Green function and in a cancellation mechanism in the critical
leading edge. This naturally raises the question of what remains of the
Bramson mechanism on discrete spaces whose large-scale geometry is no longer
one-dimensional.

Homogeneous trees provide a natural framework to address the question above.  Let
\(\TT_{q+1}\) denote the homogeneous tree in which every vertex has degree
\(q+1\).  When \(q=1\), the tree \(\TT_2\) corresponds to the lattice \(\ZZ\); for \(q\geq2\),
the volume of balls grows exponentially.  Consequently, migration can
transport mass through an exponentially increasing number of vertices and
may overcome local population growth. It is worthwhile to mention that the resulting competition between
reaction and geometry has no analogue on \(\ZZ\).

Reaction--diffusion equations on infinite trees have also been studied in
several works from both theoretical and applied perspectives; see, for
instance,
\cite{Hupkes-Jukic-Stehlik-Sigler-2023,
Besse-Faye-2021-BML,
Hoffman-Holzer-2019}
and the references therein. We highlight the work of Hoffman and Holzer \cite{Hoffman-Holzer-2019}, who investigated
Fisher--KPP dynamics on homogeneous trees through a dispersion-relation
approach. Among other results, they identified a threshold separating the
propagation and extinction regimes.

The main purpose of this paper is to deepen our understanding of the
long-time behavior of Fisher--KPP dynamics on infinite homogeneous trees,
with particular emphasis on the interplay between the branching geometry
and front propagation, and on the precise asymptotic location of the
resulting invasion fronts.

Our first contribution concerns the linearized equation. By exploiting the
radial symmetry of the tree, we derive an explicit representation of its
fundamental solution in terms of modified Bessel functions, together with
sharp global pointwise estimates. These estimates recover the propagation
threshold from the behavior of the linear kernel and, in the propagative
regime, lead naturally to the critical spreading speed \(c_\ast\) and the
associated exponential decay rate \(\lambda_\ast\).

For the nonlinear equation, our main result establishes that, for every
fixed level \(\theta\in(0,1)\), the position of the front satisfies
\[
\kappa_\theta(t)
=
c_\ast t-\frac{3}{2\lambda_\ast}\log t+O(1),
\qquad t\to\infty.
\]
Thus, the linearized problem determines not only the leading-order
propagation speed, but also the critical exponential scale that enters the
logarithmic correction.

Although radial symmetry reduces the dynamics on the tree to a
one-dimensional system, the resulting problem is not the standard
Fisher--KPP equation on \(\ZZ\). The branching structure gives rise to an
asymmetric radial dynamics, while the root satisfies a different evolution
equation. A substantial part of the analysis is therefore devoted to
understanding how these geometric features affect the propagation
mechanism and to showing that the logarithmic correction is ultimately
governed by the dynamics in the leading edge of the propagating front.

\section*{Main Results}\label{Section:Main:Results}

Our first result concerns the fundamental solution of the linearized
Fisher--KPP equation around the unstable equilibrium \(u=0\),
\begin{equation}\label{Eq:linearized-Fisher-KPP-tree}
\partial_t u
=
\alpha\Delta_{\TT_{q+1}}u+\beta u,
\end{equation}
where \(\alpha,\beta>0\) and \(\Delta_{\TT_{q+1}}\) denotes the unnormalized graph
Laplacian.

For later use, we distinguish the fundamental solution centered at the
root from the general fundamental solution. Let $p_t(x,y)$ denote the heat
kernel generated by $\alpha\Delta_{\TT_{q+1}}$, and define
\[
\mathcal G(t,x,y):=e^{\beta t}p_t(x,y),
\qquad t>0,\quad x,y\in\TT_{q+1}.
\]
By homogeneity of the tree, $\mathcal G(t,x,y)$ depends only on $d(x,y)$.
The solution considered below, corresponding to the initial datum $\delta_o$,
is therefore $u(t,x)=\mathcal G(t,x,o)$.

\begin{theorem}\label{Thm:two-sided-homogeneous}
Let \(q\geq2\) and \(\alpha,\beta>0\). Let \(u\) be the fundamental
solution of \eqref{Eq:linearized-Fisher-KPP-tree}. Then there exist
constants \(C_1,C_2>0\) such that
\[
C_1
\frac{|x|+1}
{\left(1+|x|^2+4q\alpha^2t^2\right)^{3/4}}
\exp\left\{
t\Phi\left(\frac{|x|}{t}\right)
\right\}
\leq u(t,x)
\leq
C_2
\frac{|x|+1}
{\left(1+|x|^2+4q\alpha^2t^2\right)^{3/4}}
\exp\left\{
t\Phi\left(\frac{|x|}{t}\right)
\right\},
\]
for every \(t>0\) and \(x\in\TT_{q+1}\), where
\[
|x|:=d(o,x),
\]
$o$ denotes the root of the tree
and
\[
\Phi(c)
:=
\beta-\alpha(q+1)
+
\sqrt{c^2+4q\alpha^2}
+
c\log\left(
\frac{2\alpha}
{c+\sqrt{c^2+4q\alpha^2}}
\right),
\qquad c\geq0.
\]
\end{theorem}

\begin{remark}
Theorem~\ref{Thm:two-sided-homogeneous} shows that the rate function
\(\Phi\) governs the large-time behavior of the linearized solution along
ballistic trajectories \(d(o,x)\sim ct\). In particular, its sign
determines the transition between exponential growth and decay.

Of particular interest is the regime
\[
\beta>\alpha(\sqrt q-1)^2,
\]
in which the function \(\Phi\) has a unique positive zero \(c_\ast\) such that
\[
\Phi(c)>0 \quad \text{for } 0<c<c_\ast,
\qquad \text{and}\qquad 
\Phi(c)<0 \quad \text{for } c>c_\ast.
\]
Thus, \(c_\ast\) naturally emerges from the linearized problem as the
critical spreading speed.
\end{remark}

Although \(c_\ast\) is characterized as the unique positive zero of the
rate function \(\Phi\), it also admits the following equivalent characterization
through the dispersion relation. This second point of view is particularly
useful for the nonlinear analysis and for the derivation of the Bramson
logarithmic correction.

\begin{proposition}\label{prop:critical-speed}
Let \(q\geq2\) and \(\alpha,\beta>0\), and assume that
\[
\beta>\alpha(\sqrt q-1)^2.
\]
Then 
\[
c_\ast
=
\alpha\left(
e^{\lambda_\ast}-qe^{-\lambda_\ast}
\right),
\]
where \(\lambda_\ast\in(\log(\sqrt q),\infty)\) is the unique solution of
\[
\beta-\alpha(q+1)
+
\alpha\left(
e^{\lambda_\ast}+qe^{-\lambda_\ast}
\right)
=
\alpha\lambda_\ast
\left(
e^{\lambda_\ast}-qe^{-\lambda_\ast}
\right).
\] 
Equivalently,
\[
c_\ast
=
\inf_{\lambda>0}
\frac{
\beta-\alpha(q+1)
+
\alpha\left(
e^\lambda+qe^{-\lambda}
\right)
}{\lambda},
\]
and the infimum is attained uniquely at \(\lambda=\lambda_\ast\).
\end{proposition}
\begin{remark}
Proposition~\ref{prop:critical-speed} also holds for \(q=1\), corresponding
to the homogeneous tree \(\TT_2\simeq\ZZ\). In this case the condition
\[
\beta>\alpha(\sqrt q-1)^2
\]
reduces to \(\beta>0\), and the same characterization holds with
\(\lambda_\ast\in(0,\infty)\) and \(q=1\) in the formulas above. Hence,
the proposition recovers the usual critical-speed characterization for
the Fisher--KPP equation on \(\ZZ\).
\end{remark}

We are now in a position to formulate our main result concerning the
nonlinear Fisher--KPP equation
\begin{equation}\label{Eq:FKPP:on:Trees}
\partial_t u
=
\alpha\Delta_{\TT_{q+1}}u
+
\beta u(1-u),
\qquad
\alpha,\beta>0.
\end{equation}

For each \(\theta\in(0,1)\), and for sufficiently large \(t\), we denote by
\[
\kappa_\theta(t)
:=
\sup\left\{
d(o,x):
x\in\TT_{q+1},\
u(t,x)\geq\theta
\right\},
\]
the position of the corresponding level set.

\begin{theorem}\label{Thm:Bramson}
Let \(q\geq2\) and \(\alpha,\beta>0\), and assume that
\[
\beta>\alpha(\sqrt q-1)^2.
\]
Let \(u\) be the solution of \eqref{Eq:FKPP:on:Trees} corresponding to a
nontrivial, compactly supported radial initial datum \(u_0\) satisfying
\[
0\leq u_0\leq1.
\]
Then, for every \(\theta\in(0,1)\),
\[
\kappa_\theta(t)
=
c_\ast t
-
\frac{3}{2\lambda_\ast}\log t
+
O(1),
\qquad
t\to\infty,
\]
where \(c_\ast\) and \(\lambda_\ast\) are the critical quantities
introduced in Proposition \ref{prop:critical-speed}.
\end{theorem}

\begin{remark}
It is worth noting that Theorem~\ref{Thm:Bramson} does not address the case
\[
\beta\leq\alpha(\sqrt q-1)^2.
\]

In this regime, propagation does not occur. By the comparison principle, the
nonlinear solution is bounded above by the solution of the corresponding
linearized equation, which converges uniformly to zero as
\(t\to\infty\). Therefore, the analysis of the front position is relevant
only in the propagative regime
\[
\beta>\alpha(\sqrt q-1)^2.
\]
\end{remark}

\subsection*{Organization of the paper}
The paper is organized as follows. Section~\ref{Section:FKPP:on:T:q+1}
introduces the Fisher--KPP equation on the homogeneous tree
\(\TT_{q+1}\), the graph Laplacian, and its radial formulation.

Section~\ref{sec:linear-dynamics} is devoted to the linearized problem.
We derive an explicit representation of the fundamental solution in
terms of modified Bessel functions, establish sharp global pointwise
estimates, identify the propagation threshold, and characterize the
critical pair \((c_\ast,\lambda_\ast)\).

Section~\ref{sec:nonlinear-dynamics} treats the nonlinear problem.
We first discuss extinction below the threshold and then analyze the
critical leading edge through an exponential conjugation and a tilted
Green function. We derive the dipole scale, transfer it to the
nonlinear solution by moving-domain comparison arguments, analyze the
critical traveling front, and finally obtain the Bramson logarithmic
correction.

The appendices contain the auxiliary Bessel-function estimates, the
well-posedness and comparison results for the nonlinear problem, and
the moving-domain comparison principles and barrier constructions used
in the leading-edge analysis.

\section{Fisher--KPP equation on $\mathbb{T}_{q+1}$}{\label{Section:FKPP:on:T:q+1}}

Before beginning our analysis, we explain how a Fisher--KPP equation on a
homogeneous tree arises naturally from a population model based on random
walks. The model combines migration, reproduction, and local competition.

Let \(u(t,x)\) denote the population density at the vertex \(x\) at
time \(t\). We assume that each individual moves along every incident edge
at rate \(\alpha>0\). Since every vertex has \(q+1\) neighbors, the total
jump rate is \(\alpha(q+1)\), and conditioned on a jump each neighboring
vertex is chosen with probability \(1/(q+1)\). Therefore, over a short time
interval \(h>0\),
\[
u(t+h,x)
=
u(t,x)
+
\alpha h\sum_{y\sim x}\bigl(u(t,y)-u(t,x)\bigr)
+
o(h).
\]
Dividing by \(h\) and letting \(h\to0\) gives the migration equation
\[
\partial_tu(t,x)
=
\alpha\sum_{y\sim x}\bigl(u(t,y)-u(t,x)\bigr).
\]

This motivates the unnormalized graph Laplacian
\[
\Delta_{\TT_{q+1}}f(x)
:=
\sum_{y\sim x}\bigl(f(y)-f(x)\bigr),
\]
where \(y\sim x\) means that \(d(x,y)=1\). With this convention,
\(\Delta_{\TT_{q+1}}\) is the infinitesimal generator of the
continuous-time symmetric random walk and is nonpositive on
\[
\ell^2\left(\mathbb{T}_{q+1}\right)
:=
\left\{
f:\mathbb{T}_{q+1}\to\mathbb{C}
\;: 
\sum_{x\in\mathbb{T}_{q+1}} |f(x)|^2<\infty
\right\}\, \mbox{equipped with}\,\, \|f\|_{\ell^2(\mathbb{T}_{q+1})}
:=
\left(
\sum_{x\in\mathbb{T}_{q+1}} |f(x)|^2
\right)^{1/2}.
\]

At each vertex we also assume logistic growth with intrinsic rate
\(\beta>0\), normalized carrying capacity \(1\), and reaction term
\(\beta u(1-u)\). Combining migration and local growth yields
\begin{equation}
\label{eq:Fisher-KPP-tree}
\partial_t u(t,x)
=
\alpha\Delta_{\TT_{q+1}}u(t,x)
+
\beta u(t,x)\bigl(1-u(t,x)\bigr),
\qquad
\alpha,\beta>0.
\end{equation}
Thus \eqref{eq:Fisher-KPP-tree} is the natural analogue on
\(\TT_{q+1}\) of the classical Fisher--KPP equation on the real line.
The essential geometric difference is that the number of vertices at
distance \(n\) from a fixed origin grows exponentially with \(n\). This
volume growth competes with reaction and migration and is responsible for
the propagation--extinction dichotomy discussed below.

\begin{remark}[Choice of Laplacian convention]
We adopt the unnormalized graph Laplacian in
\eqref{eq:Fisher-KPP-tree} for two main reasons. First, this convention
agrees with that used by Hoffman and Holzer
\cite{Hoffman-Holzer-2019}. Second, it is the natural generator associated
with the migration mechanism described above, in which an individual
moves along each incident edge at rate $\alpha$. Moreover, when $q=1$,
the homogeneous tree $\TT_2$ is naturally identified with $\mathbb Z$,
and our convention reduces to the standard discrete Laplacian
\[
\Delta_{\TT_2}f(n)
=
f(n-1)-2f(n)+f(n+1),
\]
as used, for instance, in
\cite{Besse-Faye-Roquejoffre-Zhang-2023-TAMS}.

A different convention, frequently adopted in the literature on random
walks and analysis on graphs (see, e.g.,
\cite{Cowling-Meda-Setti-2000-TAMS}), is to use the nonnegative {\em normalized}
Laplacian
\[
\mathcal L_{\TT_{q+1}}f(x)
:=
f(x)
-
\frac{1}{q+1}
\sum_{y\sim x}f(y).
\]
Since $\TT_{q+1}$ is $(q+1)$-regular, the normalized and unnormalized
Laplacians differ only by a constant factor:
\[
-\Delta_{\TT_{q+1}}
=
(q+1)\mathcal L_{\TT_{q+1}}.
\]
Therefore, the Fisher--KPP equation
\[
\partial_tu
=
\alpha\Delta_{\TT_{q+1}}u
+
\beta u(1-u),
\]
may equivalently be written in the normalized convention as
\[
\partial_tu
+
\alpha_{\mathrm{norm}}
\mathcal L_{\TT_{q+1}}u
=
\beta u(1-u),
\qquad
\alpha_{\mathrm{norm}}=(q+1)\alpha.
\]

Thus, on a homogeneous tree, the choice between these two conventions is
essentially a matter of normalization. Results formulated using one
Laplacian can be translated directly into the other through the
substitution
\[
\alpha_{\mathrm{norm}}=(q+1)\alpha.
\]
Throughout this paper, we use the unnormalized convention in order to
remain consistent with the random-walk interpretation above and with the
standard one-dimensional discrete Fisher--KPP equation when $q=1$.
\end{remark}

\section{Linearized Dynamics: Fundamental Solution, Pointwise Estimates, and Critical Parameters}
\label{sec:linear-dynamics}

We recall that throughout this section we consider the linearized equation
\begin{equation}\label{Eq:linearized-Fisher-KPP-tree:2}
\begin{cases}
\partial_t u(t,x)
=
\alpha \Delta_{\TT_{q+1}}u(t,x)
+\beta u(t,x),
& t>0,\quad x\in\TT_{q+1},
\\[1mm]
u(0,x)=\delta_o(x),
& x\in\TT_{q+1}.
\end{cases}
\end{equation}

In order to prove Theorem~\ref{Thm:two-sided-homogeneous}, we first derive
an explicit representation of the fundamental solution of
\eqref{Eq:linearized-Fisher-KPP-tree:2}. This representation is of independent
interest and will serve as the starting point for the sharp pointwise
estimates used throughout the subsequent analysis.

\begin{proposition}
\label{prop:fundamental-solution-homogeneous}
Let \(q\geq2\) and \(\alpha,\beta>0\). The solution of
\eqref{Eq:linearized-Fisher-KPP-tree} depends only on the distance and is given by
\begin{equation}
\label{eq:fundamental-solution-homogeneous-positive}
u(t,x)
=
\frac{
q^{-(n+1)/2}e^{(\beta-\alpha(q+1))t}
}{
\alpha t
}
\sum_{j=0}^{\infty}
q^{-j}(n+2j+1)
I_{n+2j+1}
\left(2\alpha\sqrt q\,t\right),
\end{equation}
where $n=d(o,x)$ and  \(I_m\) denotes the modified Bessel function of the first kind.
\end{proposition}

\begin{proof}
Since \(\Delta_{\TT_{q+1}}\) is invariant under graph automorphisms and
the initial datum \(\delta_o\) is invariant under every automorphism
fixing \(o\), uniqueness implies that the fundamental solution is invariant
under the stabilizer of \(o\). This stabilizer acts transitively on every
geodesic sphere
\[
S_n(o):=\{x\in\TT_{q+1}:d(o,x)=n\}.
\]
Therefore, \(u(t,x)\) depends only on \(n=d(o,x)\), and we may define
\[
h_n(t):=u(t,x),
\qquad
d(o,x)=n,
\qquad
n\in\NN_0.
\]

For \(n\geq1\), every vertex at distance \(n\) from \(o\) has one
neighbor at distance \(n-1\) and \(q\) neighbors at distance \(n+1\).
Consequently,
\[
\partial_t h_n(t)
=
\alpha h_{n-1}(t)
+\alpha qh_{n+1}(t)
+\bigl(\beta-\alpha(q+1)\bigr)h_n(t).
\]
At the root, all \(q+1\) neighboring vertices belong to the first
geodesic sphere, and hence
\[
\partial_t h_0(t)
=
\alpha(q+1)h_1(t)
+\bigl(\beta-\alpha(q+1)\bigr)h_0(t).
\]
Thus, the sequence \(\{h_n\}_{n\in\NN_0}\) satisfies the infinite system
\begin{equation}\label{Radial:System}
\begin{cases}
\partial_t h_0(t)
=
\alpha(q+1)h_1(t)
+\bigl(\beta-\alpha(q+1)\bigr)h_0(t),
&t>0,
\\[1mm]
\partial_t h_n(t)
=
\alpha h_{n-1}(t)
+\alpha qh_{n+1}(t)
+\bigl(\beta-\alpha(q+1)\bigr)h_n(t),
&t>0,\quad n\geq1,
\\[1mm]
h_n(0)=\delta_{n0},
&n\in\NN_0.
\end{cases}
\end{equation}

For \(\lambda>\beta\), define
\[
H_n(\lambda)
:=
\int_0^\infty e^{-\lambda t}h_n(t)\,dt.
\]
By Lemma~\ref{Lemma:radial-linear-bound} below,
\[
0\leq h_n(t)\leq e^{\beta t},
\qquad
t\geq0,\quad n\in\NN_0.
\]
Therefore, \(H_n(\lambda)\) is well defined for every \(\lambda>\beta\),
and
\[
0\leq H_n(\lambda)
\leq
\frac{1}{\lambda-\beta}.
\]
Taking the Laplace transform of the equation at the root, we obtain
\[
\lambda H_0(\lambda)-1
=
\alpha(q+1)H_1(\lambda)
+\bigl(\beta-\alpha(q+1)\bigr)H_0(\lambda),
\]
and therefore
\begin{equation}
\label{eq:Laplace-root-unnormalized}
\bigl(\lambda-\beta+\alpha(q+1)\bigr)H_0(\lambda)
-\alpha(q+1)H_1(\lambda)
=1.
\end{equation}
For \(n\geq1\), we similarly obtain
\begin{equation}
\label{eq:Laplace-bulk-unnormalized}
qH_{n+1}(\lambda)
-\frac{\lambda-\beta+\alpha(q+1)}{\alpha}H_n(\lambda)
+H_{n-1}(\lambda)
=0.
\end{equation}

By Lemma~\ref{Lemma:Solution:Recurrence},
\[
H_n(\lambda)=H_0(\lambda)r(\lambda)^n,
\]
where
\[
r(\lambda)
=
\frac{
\lambda-\beta+\alpha(q+1)
-
\sqrt{\bigl(\lambda-\beta+\alpha(q+1)\bigr)^2
-4\alpha^2q}
}
{2\alpha q}.
\]
In particular,
\[
H_1(\lambda)=r(\lambda)H_0(\lambda).
\]

Substituting this identity into
\eqref{eq:Laplace-root-unnormalized}, we obtain
\[
\left[
\lambda-\beta+\alpha(q+1)
-\alpha(q+1)r(\lambda)
\right]H_0(\lambda)
=1.
\]
Since \(r(\lambda)\) satisfies the characteristic equation, we have that
\[
\lambda-\beta+\alpha(q+1)
=
\alpha\left(
qr(\lambda)+\frac{1}{r(\lambda)}
\right).
\]
Consequently,
\[
\begin{aligned}
\lambda-\beta+\alpha(q+1)
-\alpha(q+1)r(\lambda)
&=
\alpha\left(
qr(\lambda)+\frac{1}{r(\lambda)}
-(q+1)r(\lambda)
\right)=
\alpha\frac{1-r(\lambda)^2}{r(\lambda)}.
\end{aligned}
\]
It follows that
\[
H_0(\lambda)
=
\frac{1}{\alpha}
\frac{r(\lambda)}{1-r(\lambda)^2},
\]
and hence
\begin{equation}
\label{eq:Hn-geometric-unnormalized}
H_n(\lambda)
=
\frac{1}{\alpha}
\frac{r(\lambda)^{n+1}}{1-r(\lambda)^2}.
\end{equation}
Since \(0<r(\lambda)<1\), we finally obtain
\[
H_n(\lambda)
=
\frac{1}{\alpha}
\sum_{j=0}^\infty
r(\lambda)^{n+2j+1}.
\]
We now apply Lemma~\ref{Lemma:Laplace-Bessel-over-t}. Set
\[
s=\lambda-\beta+\alpha(q+1),
\qquad
a=2\alpha\sqrt q.
\]
Since \(\lambda>\beta\),
\[
s
>
\alpha(q+1)
\geq
2\alpha\sqrt q
=
a,
\]
so that the assumptions of
Lemma~\ref{Lemma:Laplace-Bessel-over-t} are satisfied.
Moreover, by a direct computation we have that
\[
r(\lambda)
=
\frac{1}{\sqrt q}
\frac{s-\sqrt{s^2-a^2}}{a}.
\]
Hence, for every \(m\geq1\),
\[
\begin{aligned}
r(\lambda)^m
&=
m q^{-m/2}
\int_0^\infty
\frac{I_m(2\alpha\sqrt q\,t)}{t}
e^{-(\lambda-\beta+\alpha(q+1))t}\,dt
\\
&=
\int_0^\infty e^{-\lambda t}
\left[
\frac{m}{t}q^{-m/2}
e^{(\beta-\alpha(q+1))t}
I_m(2\alpha\sqrt q\,t)
\right]dt.
\end{aligned}
\]
Applying this identity to \eqref{eq:Hn-geometric-unnormalized} with
\(m=n+2j+1\), and using Tonelli's theorem, yields
\[
H_n(\lambda)
=
\int_0^\infty e^{-\lambda t}
\Biggl[
\frac{e^{(\beta-\alpha(q+1))t}}{\alpha t}
\sum_{j=0}^\infty
(n+2j+1)q^{-(n+2j+1)/2}
I_{n+2j+1}\left(2\alpha\sqrt q\,t\right)
\Biggr]dt.
\]
By uniqueness of the Laplace transform, for every \(n\in\NN_0\) and \(t>0\),
\begin{align*}
h_n(t)
&=
\frac{e^{(\beta-\alpha(q+1))t}}{\alpha t}
\sum_{j=0}^{\infty}
(n+2j+1)q^{-(n+2j+1)/2}
I_{n+2j+1}\left(2\alpha\sqrt q\,t\right)\\
&=
\frac{q^{-(n+1)/2}e^{(\beta-\alpha(q+1))t}}{\alpha t}
\sum_{j=0}^{\infty}
q^{-j}(n+2j+1)
I_{n+2j+1}\left(2\alpha\sqrt q\,t\right).
\end{align*}
Since $u(t,x)=h_n(t)$ whenever $n=d(o,x)$, this proves
\eqref{eq:fundamental-solution-homogeneous-positive}. 
\end{proof}

\begin{remark} As we have mentioned when \(q=1\), the homogeneous tree \(\TT_{q+1}=\TT_2\) is isomorphic to the one-dimensional lattice \(\mathbb Z\). In this case, the linearized equation reads
\[
\partial_tu(t,k)
=
\alpha\bigl(
u(t,k-1)+u(t,k+1)-2u(t,k)
\bigr)
+
\beta u(t,k).
\]
It is well known that (cf. \cite{Bateman-1943-BAMS}) its fundamental solution is the classical discrete heat kernel
\[
u(t,k)
=
e^{(\beta-2\alpha)t}I_{|k|}(2\alpha t).
\]
Such an expression is recovered from \eqref{eq:fundamental-solution-homogeneous-positive}. Indeed, setting \(q=1\), we have that 
\[
u(t,k)
=
\frac{e^{(\beta-2\alpha)t}}{\alpha t}
\sum_{j=0}^{\infty}
(|k|+2j+1)
I_{|k|+2j+1}(2\alpha t).
\]
To simplify the series, we use the recurrence relation
\[
mI_m(z)
=
\frac{z}{2}
\bigl(I_{m-1}(z)-I_{m+1}(z)\bigr),
\qquad m\geq1.
\]
Therefore, the series is telescopic:
\begin{align*}
\sum_{j=0}^{\infty}
(|k|+2j+1)I_{|k|+2j+1}(z)
&=
\frac{z}{2}
\sum_{j=0}^{\infty}
\left(
I_{|k|+2j}(z)-I_{|k|+2j+2}(z)
\right)=
\frac{z}{2}I_{|k|}(z),
\end{align*}
where we have used $\displaystyle\lim_{m\to\infty}I_m(z)=0$ for every fixed \(z>0\). Taking \(z=2\alpha t\), it follows that
\[
\sum_{j=0}^{\infty}
(n+2j+1)I_{n+2j+1}(2\alpha t)
=
\alpha t\,I_n(2\alpha t),
\]
which in turn gives
\[
u(t,k)
=
e^{(\beta-2\alpha)t}I_n(2\alpha t)
=
e^{(\beta-2\alpha)t}I_{|k|}(2\alpha t).
\]
\end{remark}

\subsection*{Proof of Theorem~\ref{Thm:two-sided-homogeneous}} Throughout the proof, we write
\[
n=|x|=d(o,x),
\qquad
\rho_\alpha:=2\alpha\sqrt q.
\]
By Lemma~\ref{Lemma:Bessel-series-comparison},
\[
(n+1)I_{n+1}(\rho_\alpha t)
\leq
\sum_{j=0}^{\infty}
q^{-j}(n+2j+1)I_{n+2j+1}(\rho_\alpha t)
\leq
\frac{q(q+1)}{(q-1)^2}(n+1)I_{n+1}(\rho_\alpha t).
\]
Hence, if
\[
\mathcal K_n(t)
:=
q^{-(n+1)/2}
e^{(\beta-\alpha(q+1))t}
(n+1)\frac{I_{n+1}(\rho_\alpha t)}{t},
\]
then Proposition~\ref{prop:fundamental-solution-homogeneous} gives
\begin{equation}\label{eq:proof-two-sided-reduction}
\frac{1}{\alpha}\mathcal K_n(t)
\leq
u(t,x)
\leq
\frac{q(q+1)}{\alpha(q-1)^2}\mathcal K_n(t).
\end{equation}

By Corollary~\ref{Cor:Bessel-nplusone-estimate},
\[
\frac{\rho_\alpha c_B}{2+\sqrt2}
\frac{e^{\Theta(n,t)}}
{\left(1+n^2+\rho_\alpha^2t^2\right)^{3/4}}
\leq
\frac{1}{t}I_{n+1}(\rho_\alpha t)
\leq
2\rho_\alpha C_B
\frac{e^{\Theta(n,t)}}
{\left(1+n^2+\rho_\alpha^2t^2\right)^{3/4}},
\]
where
\[
\Theta(n,t)
:=
\sqrt{n^2+\rho_\alpha^2t^2}
+
n\log\left(
\frac{\rho_\alpha t}
{n+\sqrt{n^2+\rho_\alpha^2t^2}}
\right).
\]
Using
\[
q^{-(n+1)/2}
=
q^{-1/2}e^{-\frac n2\log q},
\qquad
\frac{\rho_\alpha}{\alpha\sqrt q}=2,
\]
we obtain
\[
\frac{2c_B}{2+\sqrt2}
\frac{n+1}{\left(1+n^2+\rho_\alpha^2t^2\right)^{3/4}}
e^{\Psi(n,t)}
\leq
u(t,x)
\leq
\frac{4q(q+1)C_B}{(q-1)^2}
\frac{n+1}{\left(1+n^2+\rho_\alpha^2t^2\right)^{3/4}}
e^{\Psi(n,t)},
\]
with
\[
\Psi(n,t)
:=
(\beta-\alpha(q+1))t+\Theta(n,t)-\frac n2\log q.
\]

Now set \(c=n/t\). Since \(\rho_\alpha=2\alpha\sqrt q\),
\[
\Psi(n,t)
=
t\Phi(c),
\]
where
\[
\Phi(c)
:=
\beta-\alpha(q+1)
+
\sqrt{c^2+4q\alpha^2}
+
c\log\left(
\frac{2\alpha}
{c+\sqrt{c^2+4q\alpha^2}}
\right).
\]
Since \(\rho_\alpha^2=4q\alpha^2\), it follows that
\[
C_1
\frac{n+1}
{\left(1+n^2+4q\alpha^2t^2\right)^{3/4}}
e^{\,t\Phi(n/t)}
\leq
u(t,x)
\leq
C_2
\frac{n+1}
{\left(1+n^2+4q\alpha^2t^2\right)^{3/4}}
e^{\,t\Phi(n/t)},
\]
where
\[
C_1=\frac{2c_B}{2+\sqrt2},
\qquad
C_2=\frac{4q(q+1)}{(q-1)^2}C_B.
\]
This proves the result. \hfill\(\blacksquare\)

\begin{corollary}[General linear kernel]
\label{cor:general-linear-kernel}
Let $\mathcal G(t,x,y)$ be the fundamental solution of the
linearized equation. There exist constants $C_1,C_2>0$, depending only on
$\alpha$, $\beta$, and $q$, such that the two-sided estimate in
Theorem~\ref{Thm:two-sided-homogeneous} holds with $|x|$ replaced by
$d(x,y)$, uniformly for $t>0$ and $x,y\in\TT_{q+1}$.
\end{corollary}

\begin{proof}
Choose an automorphism $\varphi$ of $\TT_{q+1}$ such that
$\varphi(y)=o$. The invariance of the graph Laplacian gives
\[
\mathcal G(t,x,y)=\mathcal G(t,\varphi(x),o),
\qquad d(\varphi(x),o)=d(x,y).
\]
The conclusion follows from Theorem~\ref{Thm:two-sided-homogeneous}.
\end{proof}

\begin{remark}
The function \(\Phi\) has two different behaviors depending on the
parameters \(\alpha\), \(\beta\), and \(q\). Indeed, a direct computation
gives
\[
\Phi'(c)
=
\log\left(
\frac{2\alpha}
{c+\sqrt{c^2+4q\alpha^2}}
\right),
\qquad c\geq0.
\]
Since \(q\geq2\),
\[
c+\sqrt{c^2+4q\alpha^2}>2\alpha,
\qquad c\geq0,
\]
and hence
\[
\Phi'(c)<0,
\qquad c\geq0.
\]
Thus, \(\Phi\) is strictly decreasing on \([0,\infty)\). Moreover,
\[
\Phi(0)
=
\beta-\alpha(\sqrt q-1)^2.
\]
Therefore, if
\[
\beta\leq\alpha(\sqrt q-1)^2,
\]
then
\[
\Phi(c)<0,
\qquad c>0.
\]

On the other hand, if
\[
\beta>\alpha(\sqrt q-1)^2,
\]
then \(\Phi(0)>0\). Since \(\Phi\) is strictly decreasing and
\[
\lim_{c\to\infty}\Phi(c)=-\infty,
\]
there exists a unique \(c_\ast>0\) such that
\[
\Phi(c_\ast)=0,
\]
with
\[
\Phi(c)>0 \quad \text{for } 0\leq c<c_\ast,
\qquad
\Phi(c)<0 \quad \text{for } c>c_\ast.
\]

Thus,
\[
\beta=\alpha(\sqrt q-1)^2,
\]
is the threshold separating decay along every positive ballistic
trajectory \(d(o,x)\sim ct\) from the regime in which a positive
critical propagation speed \(c_\ast\) emerges.
\end{remark}

\begin{proof}[Proof of Proposition~\ref{prop:critical-speed}]
Define
\[
D(\lambda)
:=
\beta-\alpha(q+1)
+
\alpha\left(
e^\lambda+qe^{-\lambda}
\right),
\qquad \lambda>0,
\]
and set
\[
c(\lambda):=\frac{D(\lambda)}{\lambda},
\qquad \lambda>0.
\]
Moreover, define
\[
F(\lambda)
:=
D(\lambda)-\lambda D'(\lambda),
\qquad \lambda>0.
\]
A direct computation gives
\[
F'(\lambda)
=
-\lambda\alpha
\left(
e^\lambda+qe^{-\lambda}
\right)
<0,
\]
so that \(F\) is strictly decreasing on \((0,\infty)\). Moreover,
\[
D'(\log(\sqrt q))=0,
\]
and therefore
\[
F(\log(\sqrt q))
=
D(\log(\sqrt q))
=
\beta-\alpha(q+1)+2\alpha\sqrt q
=
\beta-\alpha(\sqrt q-1)^2
>0.
\]
On the other hand,
\[
F(\lambda)
=
\beta-\alpha(q+1)
+
\alpha(1-\lambda)e^\lambda
+
\alpha q(1+\lambda)e^{-\lambda},
\]
and hence
\[
F(\lambda)\to-\infty
\quad\text{as }\lambda\to\infty.
\]
Thus, there exists a unique
\[
\lambda_\ast>\log(\sqrt q),
\]
such that
\[
F(\lambda_\ast)=0.
\]
Since
\[
F(\lambda_\ast)
=
D(\lambda_\ast)
-
\lambda_\ast D'(\lambda_\ast)=0,
\]
this identity is equivalent to 
\[
\beta-\alpha(q+1)
+
\alpha\left(
e^{\lambda_\ast}+qe^{-\lambda_\ast}
\right)
=
\alpha\lambda_\ast
\left(
e^{\lambda_\ast}-qe^{-\lambda_\ast}
\right).
\]

We next relate \(\lambda_\ast\) to the critical speed. For
\(\lambda>\log(\sqrt q)\), set
\[
\mu
:=
\lambda-\frac12\log q.
\]
Then
\[
e^\lambda=\sqrt q\,e^\mu,
\qquad
qe^{-\lambda}=\sqrt q\,e^{-\mu},
\]
and consequently
\[
D(\lambda)
=
\beta-\alpha(q+1)
+
2\alpha\sqrt q\,\cosh\mu,
\]
while
\[
D'(\lambda)
=
2\alpha\sqrt q\,\sinh\mu.
\]

Set
\[
\xi:=D'(\lambda)
=
2\alpha\sqrt q\,\sinh\mu.
\]
Since \(\mu>0\), we have \(\xi>0\), and
\[
\mu
=
\operatorname{arsinh}
\left(
\frac{\xi}{2\alpha\sqrt q}
\right).
\]
Moreover,
\[
2\alpha\sqrt q\,\cosh\mu
=
\sqrt{\xi^2+4\alpha^2q}.
\]
Thus,
\[
\lambda
=
\operatorname{arsinh}
\left(
\frac{\xi}{2\alpha\sqrt q}
\right)
+
\frac12\log q.
\]
It follows that
\[
\begin{aligned}
F(\lambda)
&=
\beta-\alpha(q+1)
+
\sqrt{\xi^2+4\alpha^2q}
-
\xi\operatorname{arsinh}
\left(
\frac{\xi}{2\alpha\sqrt q}
\right)
-
\frac \xi2\log q=
\Phi(\xi).
\end{aligned}
\]
Hence, $F(\lambda_\ast)=0$ if and only if $\Phi\left(D'(\lambda_\ast)\right)=0$. Since
\[
\lambda_\ast>\frac12\log q,
\]
we have
\[
D'(\lambda_\ast)>0.
\]
By the uniqueness of the positive zero of \(\Phi\), it follows that
\[
c_\ast=D'(\lambda_\ast).
\]
Therefore,
\[
c_\ast
=
\alpha\left(
e^{\lambda_\ast}-qe^{-\lambda_\ast}
\right).
\]
Moreover, since
\[
\Phi'(c)
=
\log\left(
\frac{2\alpha}
{c+\sqrt{c^2+4q\alpha^2}}
\right),
\]
we obtain the equivalent characterization
\[
\lambda_\ast
=
-\Phi'(c_\ast)
=
\log\left(
\frac{c_\ast+\sqrt{c_\ast^2+4q\alpha^2}}
{2\alpha}
\right).
\]

Finally,
\[
\frac{d}{d\lambda}
\left(
\frac{D(\lambda)}{\lambda}
\right)
=
\frac{\lambda D'(\lambda)-D(\lambda)}{\lambda^2}
=
-\frac{F(\lambda)}{\lambda^2}.
\]
Since \(F\) is strictly decreasing and vanishes uniquely at
\(\lambda_\ast\),
\[
F(\lambda)>0
\quad\text{for }0<\lambda<\lambda_\ast,
\qquad
F(\lambda)<0
\quad\text{for }\lambda>\lambda_\ast.
\]
Therefore, $\lambda\mapsto\frac{D(\lambda)}{\lambda}$ is strictly decreasing on \((0,\lambda_\ast)\) and strictly increasing
on \((\lambda_\ast,\infty)\). Hence its unique global minimum is attained
at \(\lambda_\ast\). Since
\[
F(\lambda_\ast)=0,
\]
we have
\[
D(\lambda_\ast)
=
\lambda_\ast D'(\lambda_\ast),
\]
and consequently
\[
\inf_{\lambda>0}
\frac{D(\lambda)}{\lambda}
=
\frac{D(\lambda_\ast)}{\lambda_\ast}
=
D'(\lambda_\ast)
=
c_\ast.
\]
\end{proof}

\section{Nonlinear Dynamics: Leading-Edge Asymptotics and the Bramson Correction}
\label{sec:nonlinear-dynamics}

We recall that throughout this section we consider the nonlinear problem
\begin{equation}\label{Eq:Fisher-KPP-tree:nonlinear}
\begin{cases}
\partial_t u(t,x)
=
\alpha \Delta_{\TT_{q+1}}u(t,x)
+\beta u(t,x)\bigl(1-u(t,x)\bigr),
& t>0,\quad x\in\TT_{q+1},
\\[1mm]
u(0,x)=u_0(x),
& x\in\TT_{q+1},
\end{cases}
\end{equation}
where \(u_0\) is a nontrivial, compactly supported radial initial datum
with respect to the root \(o\), satisfying
\[
0\leq u_0(x)\leq1,
\qquad x\in\TT_{q+1}.
\]

Our starting point is the comparison principle, which is standard for cooperative systems; see, for
instance, \cite[Chapter~3]{Smith-1995-Book}. For completeness, in Appendix \ref{Appendix:Auxiliary:Results:2} we give a short argument adapted to our setting.

As a direct consequence of the comparison principle, we obtain extinction
of solutions below the propagation threshold. This phenomenon is a
distinctive feature of the branching geometry and is markedly different
from the behavior observed on the one-dimensional lattice.

\begin{proposition}\label{prop:extinction-below-threshold}
Let \(u\) be the solution of \eqref{Eq:FKPP:on:Trees} with a nonnegative,
compactly supported initial datum \(u_0\) satisfying \(0\leq u_0\leq1\).
Assume that
\[
\beta\leq\alpha(\sqrt q-1)^2.
\]
Then
\[
\lim_{t\to\infty}
\|u(t,\cdot)\|_{\ell^\infty(\TT_{q+1})}
=
0.
\]
\end{proposition}

\begin{proof}
Let $v$ be the solution of the corresponding linearized problem
\[
\partial_t v
=
\alpha\Delta_{\TT_{q+1}}v+\beta v,
\qquad
v(0,\cdot)=u_0.
\]
By Lemma~\ref{lem:nonlinear-wellposedness-radiality}, the comparison
principle applies. Since
\[
\beta u(1-u)\leq\beta u,
\qquad
0\leq u\leq1,
\]
we obtain
\[
0\leq u(t,x)\leq v(t,x),
\qquad
t\geq0,\quad x\in\TT_{q+1}.
\]
Let \(\mathcal G(t,x,y)\) denote the fundamental solution of the
linearized equation. Since \(u_0\) is compactly supported,
\[
v(t,x)
=
\sum_{y\in\operatorname{supp}u_0}
\mathcal G(t,x,y)u_0(y).
\]
By homogeneity, \(\mathcal G(t,x,y)\) depends only on the distance between $x$ and $y$, denoted as  
\(n=d(x,y)\). 
By Theorem~\ref{Thm:two-sided-homogeneous}, whose proof relies on
Lemma~\ref{Lemma:Bessel-series-comparison},
Lemma~\ref{Lemma:Bessel-uniform-estimate}, and
Corollary~\ref{Cor:Bessel-nplusone-estimate}, there exists $C>0$ such
that
\[
\mathcal G(t,x,y)
\leq
C
\frac{n+1}
{\left(1+n^2+4q\alpha^2t^2\right)^{3/4}}
\exp\left\{
t\Phi\left(\frac nt\right)
\right\},
\qquad
n=d(x,y).
\]
Moreover, by the definition of $\Phi$ and the differentiation formula recorded in
Lemma~\ref{lem:Phi-expansion-critical},
\[
\Phi'(c)
=
\log\left(
\frac{2\alpha}
{c+\sqrt{c^2+4q\alpha^2}}
\right),
\qquad c\geq0.
\]
Then
\[
\Phi'(c)\leq
-\frac12\log q,
\qquad c\geq0.
\]
Hence,
\[
\Phi(c)
\leq
\Phi(0)-\frac c2\log q
=
-\bigl[\alpha(\sqrt q-1)^2-\beta\bigr]
-\frac c2\log q.
\]
From this, we deduce
\[
\exp\left\{
t\Phi\left(\frac nt\right)
\right\}
\leq
e^{-[\alpha(\sqrt q-1)^2-\beta]t}q^{-n/2}.
\]
Since
\[
\left(1+n^2+4q\alpha^2t^2\right)^{3/4}
\geq
(4q\alpha^2t^2)^{3/4},
\]
we obtain, for \(t>0\),
\[
\mathcal G(t,x,y)
\leq
Ct^{-3/2}
e^{-[\alpha(\sqrt q-1)^2-\beta]t}
(n+1)q^{-n/2}.
\]
Since
\[
\sup_{n\geq0}(n+1)q^{-n/2}<\infty,
\]
it follows that
\[
\sup_{x,y\in\TT_{q+1}}
\mathcal G(t,x,y)
\leq
Ct^{-3/2}
e^{-[\alpha(\sqrt q-1)^2-\beta]t}.
\]
Consequently,
\[
u(t,x)
\leq
Ct^{-3/2}
e^{-[\alpha(\sqrt q-1)^2-\beta]t}
\sum_{y\in\operatorname{supp}u_0}u_0(y).
\]
Taking the supremum over \(x\in\TT_{q+1}\) and letting \(t\to\infty\)
proves the result.
\end{proof}

\begin{remark}
If we fix \(\alpha\) and \(\beta\), then \(q\) can always be chosen
sufficiently large so that
\[
\beta\leq\alpha(\sqrt q-1)^2.
\]
Thus, increasing the branching eventually places the system in the
extinction regime.

This observation is closely related to the phenomenon described by
Hoffman and Holzer~\cite{Hoffman-Holzer-2019}. In their normalization,
the branching parameter is fixed and the diffusion coefficient is varied;
they identify a critical diffusion threshold beyond which propagation is
lost. In our parametrization, this transition is expressed by
\[
\alpha
\geq
\frac{\beta}{(\sqrt q-1)^2}.
\]
Thus, extinction may be induced either by increasing the diffusion
strength for a fixed tree or by increasing the branching while keeping
\(\alpha\) and \(\beta\) fixed.
\end{remark}

The dynamics in the complementary regime are considerably richer and,
consequently, more delicate to analyze. To study this regime, we first
introduce several auxiliary results that will be used throughout the
subsequent analysis.

By Appendix~\ref{Appendix:Auxiliary:Results:2}, radial symmetry is
preserved by the evolution. Hence, we may write
\[
u(t,x)=u_n(t),
\qquad
n=d(o,x).
\]
The corresponding radial equation is
\begin{equation}\label{eq:radial-KPP-Bramson-new}
\begin{cases}
\partial_tu_0
=
\alpha(q+1)(u_1-u_0)+\beta u_0(1-u_0),
\\[1mm]
\partial_tu_n
=
\alpha u_{n-1}
+\alpha q u_{n+1}
-\alpha(q+1)u_n
+\beta u_n(1-u_n),
& n\geq1.
\end{cases}
\end{equation}

The equation at \(n=0\) is exceptional, since it reflects the different
geometry of the root, while for every \(n\geq1\) the radial equation has
constant coefficients and is therefore translation invariant in the
spatial index. Thus, the root can be viewed as a localized defect confined
to a fixed neighborhood of \(n=0\).

The asymptotic regimes considered below, however, are observed in moving
spatial regions whose distance from the root tends to infinity as
\(t\to\infty\). More precisely, the relevant indices satisfy
\[
n\sim c_\ast t,
\]
up to lower-order corrections, and hence \(n\to\infty\). Consequently, for
all sufficiently large times, these moving regions lie entirely in the
bulk \(n\geq1\) and do not intersect the exceptional site \(n=0\).

It follows that the local structure of the equation near the root does not
enter the leading-order analysis of the propagating front. In particular,
the behavior in the leading edge is governed by the translation-invariant
bulk equation, namely the second equation in
\eqref{eq:radial-KPP-Bramson-new}. The root may influence the solution at
finite times and its overall amplitude, but it does not modify the
asymptotic mechanism determining the front propagation in the moving
region.

The equation at the root is not translation invariant, whereas the
radial equation is translation invariant for every $n\geq1$. Let
$0<\delta<1$ and consider a moving region satisfying
\[
n-c_\ast t\geq-t^\delta.
\]
Since $c_\ast>0$, there exists $T_\delta>0$ such that
\[
c_\ast t-t^\delta\geq1,
\qquad t\geq T_\delta.
\]
Thus every point of this moving region belongs to the bulk $n\geq1$ for
all $t\geq T_\delta$. Hence all moving-domain comparisons and barrier
arguments occur entirely in the translation-invariant bulk equation. The
root affects only finite-time initialization and multiplicative constants;
it does not modify $c_\ast$, $\lambda_\ast$, or the coefficient
$3/(2\lambda_\ast)$ in the logarithmic correction.

In order to analyze the behavior of the solution near the leading edge,
we introduce the critical exponential conjugation
\begin{equation}\label{eq:critical-conjugation-new}
v_n(t):=e^{\lambda_\ast(n-c_\ast t)}u_n(t),
\qquad n\geq0.
\end{equation}
The main reason for introducing \(v_n\) in this form is that the factor
\[
e^{-\lambda_\ast(n-c_\ast t)},
\]
captures the dominant exponential decay of the solution near the leading
edge. Thus, the expression \eqref{eq:critical-conjugation-new} factors out this decay and isolates the slower behavior that remains.

\begin{lemma}\label{lem:critical-conjugation-new}
For every \(n\geq1\), the function \(v\) defined by
\eqref{eq:critical-conjugation-new} satisfies
\[
\partial_t v_n
=
\mathscr L_\ast v_n
-\beta e^{-\lambda_\ast(n-c_\ast t)}v_n^2,
\]
where
\[
\mathscr L_\ast v_n
:=
a_\ast(v_{n-1}-v_n)
+
b_\ast(v_{n+1}-v_n),
\]
with
\[
a_\ast:=\alpha e^{\lambda_\ast},
\qquad
b_\ast:=\alpha q e^{-\lambda_\ast}.
\]
\end{lemma}

\begin{proof}
By a direct computation we have that
\[
\partial_tu_n
=
e^{-\lambda_\ast(n-c_\ast t)}
\left(
\partial_tv_n+\lambda_\ast c_\ast v_n
\right).
\]
Further, 
\[
u_{n+1}
=
e^{-\lambda_\ast(n-c_\ast t)}
e^{-\lambda_\ast}v_{n+1}, \quad \text{and}\quad 
u_{n-1}
=
e^{-\lambda_\ast(n-c_\ast t)}
e^{\lambda_\ast}v_{n-1}.
\]

Substituting these expressions into the bulk equation in
\eqref{eq:radial-KPP-Bramson-new}, we obtain
\[
\partial_tv_n
=
\alpha e^{\lambda_\ast}v_{n-1}
+
\alpha q e^{-\lambda_\ast}v_{n+1}
+
\bigl[
\beta-\alpha(q+1)-\lambda_\ast c_\ast
\bigr]v_n
-
\beta e^{-\lambda_\ast(n-c_\ast t)}v_n^2.
\]
Using the critical relation
\[
\lambda_\ast c_\ast
=
\beta-\alpha(q+1)
+
\alpha\left(
e^{\lambda_\ast}+qe^{-\lambda_\ast}
\right),
\]
the coefficient of \(v_n\) becomes
\[
-\alpha e^{\lambda_\ast}
-\alpha q e^{-\lambda_\ast}.
\]
Hence,
\[
\partial_tv_n
=
\alpha e^{\lambda_\ast}(v_{n-1}-v_n)
+
\alpha q e^{-\lambda_\ast}(v_{n+1}-v_n)
-
\beta e^{-\lambda_\ast(n-c_\ast t)}v_n^2.
\]
By the definition of \(a_\ast\), \(b_\ast\), and
\(\mathscr L_\ast\), this is the desired expression.
\end{proof}

\begin{remark}
The conjugated equation in
Lemma~\ref{lem:critical-conjugation-new} holds for $n\geq1$ only. The
root $n=0$ satisfies a different discrete boundary equation, due to the
fact that all its $q+1$ neighbors belong to the first sphere. This
boundary effect is irrelevant for the leading-edge analysis, whose
moving domains are contained in $n\geq1$ for all sufficiently large
times.
\end{remark}

\begin{remark}
The coefficients of the conjugated operator satisfy
\[
a_\ast-b_\ast=c_\ast,
\qquad
a_\ast+b_\ast=D''(\lambda_\ast).
\]
Indeed,
\[
D'(\lambda)
=
\alpha\left(e^\lambda-qe^{-\lambda}\right),
\qquad
D''(\lambda)
=
\alpha\left(e^\lambda+qe^{-\lambda}\right),
\]
and Proposition~\ref{prop:critical-speed} gives
\[
D'(\lambda_\ast)=c_\ast.
\]
Therefore,
\[
a_\ast-b_\ast
=
\alpha e^{\lambda_\ast}
-\alpha q e^{-\lambda_\ast}
=
c_\ast,
\]
while
\[
a_\ast+b_\ast
=
\alpha e^{\lambda_\ast}
+\alpha q e^{-\lambda_\ast}
=
D''(\lambda_\ast).
\]
In particular, \(\mathscr L_\ast\) is an asymmetric discrete diffusion
operator whose drift is precisely \(c_\ast\).
\end{remark}

Let \(\mathcal G_n(t)\) be the fundamental solution on \(\ZZ\) of
\begin{equation}\label{eq:critical-linear-whole-line-new}
\partial_tw_n=\mathscr L_\ast w_n,
\qquad
w_n(0)=\delta_{0n}.
\end{equation}

Some properties of \(\mathcal{G}_n\) are summarized in Lemma \ref{lem:tilted-Green-new}. The following result is an adaptation to the asymmetric case of the corresponding result in \cite{Besse-Faye-Roquejoffre-Zhang-2023-TAMS}.

\begin{proposition}\label{prop:Green-refined-asymptotics-new}
For every \(\vartheta>0\) and every \(\gamma\in(0,1/2)\), there exist
constants \[
T_0=T_0(\vartheta,\gamma)>0,
\qquad
C=C(\vartheta,\gamma)>0,
\qquad
\delta_0=\delta_0(\vartheta,\gamma)>0,
\]
such that, for every $t\geq T_0$ and every
$n\in\ZZ$ satisfying
\[
|n-c_\ast t|\leq\vartheta t^\gamma,
\]
one has
\begin{equation}\label{eq:Green-decomp-new}
\mathcal G_n(t)
=
\mathcal H_n(t)
+
\mathcal E_n(t),
\end{equation}
where
\begin{equation}\label{eq:H-new}
\mathcal H_n(t)
=
\left[
\frac{1}{\sqrt{s_\ast t}}
+
\frac{c_\ast}{6s_\ast^2t}(X^3-3X)
\right]\mathfrak g(X),
\end{equation}
with
\[
s_\ast:=a_\ast+b_\ast,
\qquad
X:=\frac{n-c_\ast t}{\sqrt{s_\ast t}},
\qquad
\mathfrak g(X):=\frac{1}{\sqrt{2\pi}}e^{-X^2/2},
\]
and
\begin{equation}\label{eq:E-new}
|\mathcal E_n(t)|
\leq
Ct^{-3/2}
\exp\left(
-\delta_0\frac{(n-c_\ast t)^2}{t}
\right).
\end{equation}
\end{proposition}

\begin{proof}
We use \cite[Proposition~4.1]{Besse-Faye-Roquejoffre-Zhang-2023-TAMS}
only in the moderate-deviation regime
\[
|n-c_\ast t|\leq\vartheta t^\gamma,
\qquad 0<\gamma<\frac12.
\]
Indeed, the corresponding self-similar variable tends to zero uniformly:
\[
\left|\frac{n-c_\ast t}{\sqrt{s_\ast t}}\right|
\leq \frac{\vartheta}{\sqrt{s_\ast}}t^{\gamma-1/2}\longrightarrow0.
\]
We now verify the correspondence of the coefficients under a fixed time
rescaling. Set
\[
\rho_\ast
:=
\frac12\log\frac{a_\ast}{b_\ast}
=
\lambda_\ast-\frac12\log q,
\qquad
A_\ast
:=
\sqrt{a_\ast b_\ast}
=
\alpha\sqrt q.
\]
Then
\[
a_\ast=A_\ast e^{\rho_\ast},
\qquad
b_\ast=A_\ast e^{-\rho_\ast}.
\]
Hence, under the time rescaling
\[
\tau=A_\ast t,
\]
the equation
\[
\partial_t z_n
=
a_\ast(z_{n-1}-z_n)
+
b_\ast(z_{n+1}-z_n),
\]
becomes
\[
\partial_\tau z_n
=
e^{\rho_\ast}(z_{n-1}-z_n)
+
e^{-\rho_\ast}(z_{n+1}-z_n).
\]
Equivalently,
\[
\partial_\tau z_n
=
e^{\rho_\ast}
(z_{n-1}-2z_n+z_{n+1})
-
\left(e^{\rho_\ast}-e^{-\rho_\ast}\right)
(z_{n+1}-z_n).
\]
Since
\[
e^{\rho_\ast}-e^{-\rho_\ast}
=
\frac{c_\ast}{A_\ast},
\]
we obtain
\[
\partial_\tau z_n
=
e^{\rho_\ast}
(z_{n-1}-2z_n+z_{n+1})
-
\frac{c_\ast}{A_\ast}(z_{n+1}-z_n),
\]
which is precisely the form considered in
\cite[Section 3]{Besse-Faye-Roquejoffre-Zhang-2023-TAMS}, with
\[
\lambda_\ast^{\rm BFRZ}=\rho_\ast,
\qquad
c_\ast^{\rm BFRZ}=\frac{c_\ast}{A_\ast}.
\]

It remains only to translate their variables back to the present
notation. Since
\[
c_\ast^{\rm BFRZ}\tau=c_\ast t
\]
and
\[
2\cosh(\rho_\ast)\tau
=
(a_\ast+b_\ast)t
=
s_\ast t,
\]
their self-similar variable becomes
\[
X
=
\frac{n-c_\ast t}{\sqrt{s_\ast t}}.
\]
Moreover, their first correction term becomes
\[
\frac{c_\ast}{6s_\ast^2t}
(X^3-3X)\mathfrak g(X).
\]
Finally, the remainder estimate is preserved under the rescaling
\(\tau=A_\ast t\), after modifying the constants \(C\), \(\delta_0\),
and \(T_0\).

Therefore, their expansion gives
\eqref{eq:Green-decomp-new}--\eqref{eq:E-new}.
\end{proof}

\subsection{The dipole scale}

Fix a nontrivial compactly supported odd sequence \(w^0\) on \(\ZZ\) such
that
\begin{equation}\label{eq:odd-data-new}
w_n^0=-w_{-n}^0\geq0,
\qquad n\geq1,
\end{equation}
and let \(w\) solve \eqref{eq:critical-linear-whole-line-new} with datum
\(w^0\).  If \(\operatorname{supp}w^0\subset[-J,J]\), then
\begin{equation}\label{eq:dipole-representation-new}
w_n(t)
=
\sum_{\ell=1}^Jw_\ell^0
\bigl(\mathcal G_{n-\ell}(t)-\mathcal G_{n+\ell}(t)\bigr).
\end{equation}
Set
\[
M_1:=\sum_{\ell=1}^J\ell w_\ell^0>0.
\]

\begin{lemma}[Linear dipole]\label{lem:critical-dipole-new}
Let \(0<\delta<\gamma<1/2\) and \(\vartheta>0\).  Uniformly in
\begin{equation}\label{eq:dipole-asymptotic-region-new}
t^\delta\leq y_n(t):=n-c_\ast t\leq\vartheta t^\gamma,
\end{equation}
one has
\begin{equation}\label{eq:dipole-asymptotic-new}
w_n(t)
=
\frac{2M_1}{\sqrt{2\pi}s_\ast^{3/2}}
\frac{y_n(t)}{t^{3/2}}
\exp\left(-\frac{y_n(t)^2}{2s_\ast t}\right)
\bigl(1+o(1)\bigr).
\end{equation}
In addition, for every \(L>0\) there exist \(T_L>0\) and constants
\(0<c_L<C_L\) such that
\begin{equation}\label{eq:dipole-positive-new}
w_n(t)>0
\end{equation}
and
\begin{equation}\label{eq:dipole-two-sided-new}
c_L\frac{y_n(t)}{t^{3/2}}
 e^{-C_Ly_n(t)^2/t}
\leq
w_n(t)
\leq
C_L\frac{y_n(t)}{t^{3/2}}
 e^{-y_n(t)^2/(C_Lt)}
\end{equation}
whenever
\[
t\geq T_L,
\qquad
1\leq y_n(t)\leq L\sqrt t.
\]
Finally, for every \(\gamma<1/2\),
\begin{equation}\label{eq:dipole-rough-new}
|w_n(t)|
\leq
C\frac{|y_n(t)|+1}{t^{3/2}}
\end{equation}
uniformly for \(|y_n(t)|\leq t^\gamma\) and large \(t\).
\end{lemma}

For the sake of brevity of the text, we only indicate the modifications needed with respect to the argument in
\cite[Section~4.2]{Besse-Faye-Roquejoffre-Zhang-2023-TAMS}.

\begin{proof}
By the oddness of \(w^0\),
\[
w_n(t)
=
\sum_{\ell=1}^J
w_\ell^0
\bigl(
\mathcal G_{n-\ell}(t)-\mathcal G_{n+\ell}(t)
\bigr).
\]
Set
\[
y:=y_n(t)=n-c_\ast t.
\]
Since $\ell\in\{1,\ldots,J\}$ is fixed, the points $n-\ell$ and
$n+\ell$ remain in a moderate-deviation region of the form
\[
|m-c_\ast t|\leq 2\vartheta t^\gamma
\]
for all sufficiently large $t$. Hence,
Proposition~\ref{prop:Green-refined-asymptotics-new} applies to both
$\mathcal G_{n-\ell}(t)$ and $\mathcal G_{n+\ell}(t)$.
Therefore,
\[
\mathcal G_{n-\ell}(t)-\mathcal G_{n+\ell}(t)
=
\frac{2\ell}{\sqrt{2\pi}\,s_\ast^{3/2}}
\frac{y}{t^{3/2}}
\exp\left(
-\frac{y^2}{2s_\ast t}
\right)
+
O(t^{-3/2}),
\]
uniformly for
\[
|y|\leq \vartheta t^\gamma,
\qquad
\gamma<\frac12.
\]
Indeed, the leading Gaussian term gives the displayed contribution,
whereas the first correction term and the remainder in
Proposition~\ref{prop:Green-refined-asymptotics-new} are both
\(O(t^{-3/2})\) in this region.

Therefore,
\[
w_n(t)
=
\frac{2M_1}{\sqrt{2\pi}\,s_\ast^{3/2}}
\frac{y}{t^{3/2}}
\exp\left(
-\frac{y^2}{2s_\ast t}
\right)
+
O(t^{-3/2}),
\]
where
\[
M_1=\sum_{\ell=1}^J\ell w_\ell^0>0.
\]
If
\[
t^\delta\leq y\leq\vartheta t^\gamma,
\qquad
0<\delta<\gamma<\frac12,
\]
then
\[
\frac{y^2}{t}\to0,
\]
uniformly and, since \(y\geq t^\delta\) and \(y^2/t\to0\) uniformly in the region under
consideration,
\[
\frac{
t^{-3/2}
}{
\dfrac{y}{t^{3/2}}
\exp\left(-\dfrac{y^2}{2s_\ast t}\right)
}
=
\frac1y
\exp\left(\frac{y^2}{2s_\ast t}\right)
\to0,
\]
uniformly as \(t\to\infty\). Hence
\[
w_n(t)
=
\frac{2M_1}{\sqrt{2\pi}\,s_\ast^{3/2}}
\frac{y_n(t)}{t^{3/2}}
\exp\left(
-\frac{y_n(t)^2}{2s_\ast t}
\right)
\bigl(1+o(1)\bigr),
\]
uniformly in \eqref{eq:dipole-asymptotic-region-new}.

Moreover, for every $\widetilde\gamma\in(0,1/2)$, there exists
$C_{\widetilde\gamma}>0$ such that
\[
|w_n(t)|
\leq
C_{\widetilde\gamma}
\frac{|y_n(t)|+1}{t^{3/2}},
\]
uniformly for
\[
|y_n(t)|\leq t^{\widetilde\gamma}
\]
and all sufficiently large $t$.
Thus, \eqref{eq:dipole-rough-new} is proved.

The preceding asymptotic expansion yields positivity only in the region
where $y_n(t)\to\infty$. To obtain positivity and uniform two-sided
bounds down to the scale $y_n(t)\geq1$, we use the odd-data dipole
argument of \cite[Section~4.2]{Besse-Faye-Roquejoffre-Zhang-2023-TAMS}.
After the fixed time rescaling
\[
\tau=A_\ast t,
\qquad
A_\ast=\sqrt{a_\ast b_\ast},
\]
the equation generated by $\mathscr L_\ast$ takes the form considered
there. Returning to the original time variable changes only the
positive constants. Hence, for every $L>0$, there exist $T_L>0$ and
constants $0<c_L<C_L$ such that
\[
w_n(t)>0,
\]
and
\[
c_L\frac{y_n(t)}{t^{3/2}}
\exp\left(
-C_L\frac{y_n(t)^2}{t}
\right)
\leq
w_n(t)
\leq
C_L\frac{y_n(t)}{t^{3/2}}
\exp\left(
-\frac{y_n(t)^2}{C_Lt}
\right)
\]
whenever
\[
t\geq T_L,
\qquad
1\leq y_n(t)\leq L\sqrt t.
\]
This proves the result.
\end{proof}

\begin{remark}
The restriction
\[
y_n(t)\geq t^\delta,
\]
in \eqref{eq:dipole-asymptotic-new} is deliberate. The remainder in the
local central limit expansion is only controlled at order \(t^{-3/2}\).
Therefore, obtaining a relative asymptotic when
\(y_n(t)\) remains bounded would require a higher-order expansion of the
Green function. Such a refinement is not needed in what follows.
\end{remark}

\subsection{Transfer of the dipole scale to the nonlinear leading edge}

The moving-domain comparison argument follows the strategy of
\cite{Besse-Faye-Roquejoffre-Zhang-2023-TAMS}; the modifications required
in the present setting are collected in
Appendix~\ref{Appendix:Nonlinear:Barriers}. Its consequence is the
following.

\begin{proposition}[Intermediate leading edge]
\label{prop:leading-edge-new}
Assume that $u_0$ is nontrivial, radial, nonnegative, and compactly
supported. Then there exist \(\eta\in(0,1/4)\), \(T_\eta>0\), and constants
\(0<c_0<C_0\) such that
\begin{equation}\label{eq:leading-edge-v-new}
c_0\frac{n-c_\ast t}{t^{3/2}}
\leq
v_n(t)
\leq
C_0\frac{n-c_\ast t}{t^{3/2}}, 
\end{equation}
for \(t\geq T_\eta\) and
\begin{equation}\label{eq:leading-edge-region-new}
\frac12t^\eta
\leq
n-c_\ast t
\leq
2t^\eta.
\end{equation}
Equivalently,
\begin{equation}\label{eq:leading-edge-u-new}
c_0\frac{n-c_\ast t}{t^{3/2}}
e^{-\lambda_\ast(n-c_\ast t)}
\leq
u_n(t)
\leq
C_0\frac{n-c_\ast t}{t^{3/2}}
e^{-\lambda_\ast(n-c_\ast t)}.
\end{equation}
\end{proposition}

\begin{proof}
We only indicate the points that differ from \cite{Besse-Faye-Roquejoffre-Zhang-2023-TAMS}.
The barriers constructed in
Appendix~\ref{Appendix:Nonlinear:Barriers} have as their principal
linear term the dipole \(w\) associated with \(\mathscr L_\ast\).
By Lemma~\ref{lem:critical-dipole-new},
\[
w_n(t)\asymp
\frac{n-c_\ast t}{t^{3/2}},
\]
uniformly in \eqref{eq:leading-edge-region-new}. Indeed, in this region
\(n-c_\ast t\asymp t^\eta\) and
\[
\frac{(n-c_\ast t)^2}{t}\longrightarrow0.
\]

The remaining correction terms in the barriers are of lower order in
this region. Moreover, since
\[
n=c_\ast t+O(t^\eta),
\qquad c_\ast>0,
\]
the moving domain lies entirely in the bulk \(n\geq1\) for all
sufficiently large \(t\). Thus, the exceptional equation at the root
does not enter the comparison argument.

The upper and lower barriers therefore give
\eqref{eq:leading-edge-v-new}. Finally, using
\[
u_n(t)
=
e^{-\lambda_\ast(n-c_\ast t)}v_n(t), 
\]
gives \eqref{eq:leading-edge-u-new}.
\end{proof}

\subsection{ The critical traveling front}

The translation-invariant bulk equation admits traveling fronts. We shall
only need the critical one and its behavior at the unstable equilibrium.

\begin{lemma}[Critical front and its tail]
\label{lem:critical-front-new}
Assume that
\[
\beta>\alpha(\sqrt q-1)^2.
\]
Then there exists a strictly decreasing function
\[
U_\ast\in C^1(\RR)
\]
satisfying

\begin{equation}\label{eq:critical-wave-new}
-c_\ast U_\ast'(z)
=
\alpha\bigl[
U_\ast(z-1)+qU_\ast(z+1)-(q+1)U_\ast(z)
\bigr]
+
\beta U_\ast(z)(1-U_\ast(z)),
\end{equation}
with
\[
U_\ast(-\infty)=1,
\qquad
U_\ast(+\infty)=0.
\]
Moreover, there exists \(C_\ast>0\) such that
\begin{equation}\label{eq:critical-front-tail-C-new}
U_\ast(z)
=
C_\ast z e^{-\lambda_\ast z}(1+o(1)),
\qquad z\to+\infty.
\end{equation}
After a translation of the profile, we may normalize \(C_\ast=1\).
Moreover, there exists \(C>0\) such that
\[
 |U_\ast'(z)|\leq C(1+z)e^{-\lambda_\ast z},\qquad z\geq0.
\]
\end{lemma}



\begin{proof}
Existence of a monotone front at the minimal speed follows from the
general theory for discrete monostable equations; see
\cite{AlHaj-Monneau-2024-JDE}. Indeed, the off-diagonal couplings in
\eqref{eq:critical-wave-new} are
\[
\alpha>0,
\qquad
\alpha q>0,
\]
and
\[
f(u)=\beta u(1-u),
\]
satisfies the KPP condition. In this case the minimal velocity is
linearly determined and coincides with the variational speed obtained in
Proposition~\ref{prop:critical-speed}, namely \(c_\ast\). The strong
comparison principle gives the strict monotonicity of the profile.

We next determine its behavior at \(+\infty\). We adapt the
Laplace--Tauberian argument of
\cite{Carr-Chmaj-2004-PAMS} to the present asymmetric nearest-neighbor
operator.

Set
\[
\mathcal D(\lambda)
:=
D(\lambda)-c_\ast\lambda,\quad \mbox{where}\quad D(\lambda)
:=
\beta-\alpha(q+1)
+
\alpha\left(
e^\lambda+qe^{-\lambda}
\right).
\]
By the definition of \(c_\ast\),
\[
\mathcal D(\lambda)\geq0,
\qquad
\mathcal D(\lambda)=0
\iff
\lambda=\lambda_\ast,
\]
and the criticality relation gives
\begin{equation}\label{eq:D-critical-expansion-new}
\mathcal D(\lambda)
=
\frac{s_\ast}{2}
(\lambda-\lambda_\ast)^2
+
O(|\lambda-\lambda_\ast|^3),
\end{equation}
where
\[
s_\ast
=
D''(\lambda_\ast)
=
\alpha\left(
e^{\lambda_\ast}+qe^{-\lambda_\ast}
\right)>0.
\]

The exponential estimates used in the construction of the critical wave
give
\[
U_\ast(z)\leq Ce^{-\eta z},
\qquad z\geq0,
\]
for some \(C,\eta>0\). 
Since
\[
0<U_\ast(z)\leq1,
\qquad z\leq0,
\]
and
\[
U_\ast(z)\leq Ce^{-\eta z},
\qquad z\geq0,
\]
we have that the bilateral Laplace transform
\[
\widehat U(\lambda)
:=
\int_{\RR}e^{\lambda z}U_\ast(z)\,dz,
\]
is finite for every $0<\lambda<\eta$.

Define
\[
\Lambda
:=
\sup\left\{
\lambda>0:
\widehat U(\lambda)<\infty
\right\}.
\]

For \(0<\lambda<\Lambda\), multiplying
\eqref{eq:critical-wave-new} by \(e^{\lambda z}\) and integrating gives
\begin{equation}\label{eq:front-Laplace-identity-new}
\mathcal D(\lambda)\widehat U(\lambda)
=
\beta\widehat{U^2}(\lambda),
\end{equation}
where
\[
\widehat{U^2}(\lambda)
:=
\int_{\RR}e^{\lambda z}U_\ast(z)^2\,dz.
\]
Here we used
\[
\int_{\RR}e^{\lambda z}U_\ast(z-1)\,dz
=
e^\lambda\widehat U(\lambda),
\qquad
\int_{\RR}e^{\lambda z}U_\ast(z+1)\,dz
=
e^{-\lambda}\widehat U(\lambda).
\]

We claim that
\[
\Lambda=\lambda_\ast.
\]
If \(\Lambda>\lambda_\ast\), then
\eqref{eq:front-Laplace-identity-new} at
\(\lambda=\lambda_\ast\) gives
\[
0
=
\beta\widehat{U^2}(\lambda_\ast),
\]
which is impossible.

Conversely, suppose that \(\Lambda<\lambda_\ast\). Choose
\[
\lambda_0\in(\Lambda/2,\Lambda).
\]
Since \(U_\ast\) is decreasing and
\(\widehat U(\lambda_0)<\infty\),
\[
U_\ast(z)\leq C_{\lambda_0}e^{-\lambda_0z},
\qquad z\geq1.
\]
Thus, \(\widehat{U^2}\) is analytic in a neighborhood of
\(\lambda=\Lambda\). Since
\[
\mathcal D(\Lambda)>0,
\]
identity \eqref{eq:front-Laplace-identity-new} analytically continues
\(\widehat U\) through its abscissa of convergence, contradicting the
standard Laplace-transform singularity theorem for nonnegative functions.
Hence
\[
\Lambda=\lambda_\ast.
\]

Choose now
\[
\lambda_0\in(\lambda_\ast/2,\lambda_\ast).
\]
The preceding argument shows that \(\widehat{U^2}\) is analytic in a
neighborhood of \(\lambda_\ast\). Consequently,
\[
\widehat{U^2}(\lambda)
=
\widehat{U^2}(\lambda_\ast)
+
O(\lambda-\lambda_\ast).
\]
Combining this with
\eqref{eq:front-Laplace-identity-new} and
\eqref{eq:D-critical-expansion-new} yields
\begin{equation}\label{eq:front-double-pole-new}
\widehat U(\lambda)
=
\frac{
K_\ast+O(\lambda-\lambda_\ast)
}{
(\lambda_\ast-\lambda)^2
},
\end{equation}
where
\[
K_\ast
:=
\frac{2\beta}{s_\ast}
\widehat{U^2}(\lambda_\ast)
>0.
\]
It remains to check that \(\lambda_\ast\) is the only zero of
\(\mathcal D\) on the line
\(\Re\lambda=\lambda_\ast\). If
\[
\lambda=\lambda_\ast+i\tau,
\]
then
\[
\Re\mathcal D(\lambda)
=
s_\ast(\cos\tau-1).
\]
Thus a zero must satisfy \(\tau=2\pi k\), \(k\in\ZZ\). At such a point,
\[
\Im\mathcal D(\lambda)
=
c_\ast(\sin\tau-\tau)
=
-2\pi k c_\ast,
\]
which vanishes only for \(k=0\).

Therefore the first singularity of the Laplace transform is a double
pole at \(\lambda_\ast\), with no other singularity on the corresponding
boundary line.

For completeness, we verify the analytic input in the Tauberian step.
Choose \(\lambda_0\in(\lambda_\ast/2,\lambda_\ast)\). The monotonicity
argument above yields \(U_\ast(z)\leq C e^{-\lambda_0z}\) for \(z\geq1\);
hence \(\widehat{U^2}\) is analytic and bounded on every closed vertical
substripe of \(\Re\lambda<2\lambda_0\). The identity
\eqref{eq:front-Laplace-identity-new} consequently continues
\(\widehat U\) meromorphically across \(\Re\lambda=\lambda_\ast\).
The computation of \(\mathcal D(\lambda_\ast+i\tau)\) above, together
with the fact that \(|\mathcal D(\sigma+i\tau)|\to\infty\) as
\(|\tau|\to\infty\), provides a strip about that line in which the sole
singularity is the double pole at \(\lambda_\ast\). Laplace inversion on
a vertical line to the left of the pole and a contour shift across this
pole give its residue \(K_\ast z e^{-\lambda_\ast z}\); the bounded
vertical-strip transform of \(U_\ast^2\) makes the remaining integral
\(o(ze^{-\lambda_\ast z})\).
The Tauberian theorem used in
\cite{Carr-Chmaj-2004-PAMS} gives
\[
U_\ast(z)
=
K_\ast z e^{-\lambda_\ast z}(1+o(1)),
\qquad
z\to+\infty.
\]
Thus \eqref{eq:critical-front-tail-C-new} holds with
\(C_\ast=K_\ast\).

Finally, replacing \(U_\ast(z)\) by \(U_\ast(z+s)\) changes the leading
coefficient to
\[
C_\ast e^{-\lambda_\ast s}.
\]
Hence there is a unique translation for which this coefficient equals
one. With this normalization,
\begin{equation}\label{eq:critical-front-tail-new}
U_\ast(z)
=
z e^{-\lambda_\ast z}(1+o(1)),
\qquad z\to+\infty.
\end{equation}

It remains to prove the derivative estimate used later. The preceding
asymptotic gives, for \(j\in\{-1,0,1\}\) and all sufficiently large \(z\),
\[
0<U_\ast(z+j)\leq C(1+z)e^{-\lambda_\ast z}.
\]
Solving \eqref{eq:critical-wave-new} for \(U_\ast'(z)\), and using
\(0<U_\ast<1\), therefore yields
\[
c_\ast|U_\ast'(z)|\leq C(1+z)e^{-\lambda_\ast z}.
\]
The same derivative is bounded on every compact interval by
\eqref{eq:critical-wave-new}; increasing \(C\) proves
\(|U_\ast'(z)|\leq C(1+z)e^{-\lambda_\ast z}\) for all \(z\geq0\).

For the moving-strip argument, we use the globally Lipschitz extension
\[
\widetilde f(s)=
\begin{cases}
\beta s, & s\leq0,\\[1mm]
\beta s(1-s), & 0\leq s\leq1,\\[1mm]
0, & s\geq1.
\end{cases}
\]
It agrees with the logistic nonlinearity on $[0,1]$ and permits the use of
barriers that may take negative values.
\end{proof}

\subsection{Identification of the logarithmic shift}

For $t>T$ with $T>1$, let us define
\begin{equation}\label{eq:m-new}
m(t)
:=
c_\ast t
-
\frac{3}{2\lambda_\ast}\log t,
\end{equation}
and
\begin{equation}\label{eq:z-new}
z_n(t)
:=
n-m(t)
=
n-c_\ast t
+
\frac{3}{2\lambda_\ast}\log t.
\end{equation}

The intermediate leading-edge estimate provides the following formal matching argument for the logarithmic correction. Write
\[
y:=n-c_\ast t.
\]
If the nonlinear front is located a distance \(r(t)\) behind
\(c_\ast t\), then its traveling-wave coordinate is \(y+r(t)\). Since
\[
U_\ast(z)
=
z e^{-\lambda_\ast z}(1+o(1)),
\]
for
\[
y\asymp t^\eta,
\qquad
r(t)=O(\log t),
\]
we have
\[
U_\ast(y+r(t))
=
e^{-\lambda_\ast r(t)}
y e^{-\lambda_\ast y}(1+o(1)).
\]
On the other hand, Proposition~\ref{prop:leading-edge-new} gives
\[
u_n(t)
\asymp
t^{-3/2}y e^{-\lambda_\ast y}.
\]
Therefore matching requires
\[
e^{-\lambda_\ast r(t)}
\asymp
t^{-3/2},
\]
and hence
\[
r(t)
=
\frac{3}{2\lambda_\ast}\log t+O(1).
\]

We now absorb the factor \(t^{-3/2}\) by defining
\begin{equation}\label{eq:V-new}
V_n(t):=t^{3/2}v_n(t).
\end{equation}
Since
\[
e^{-\lambda_\ast(n-c_\ast t)}
=
t^{3/2}e^{-\lambda_\ast z_n(t)},
\]
we have the exact identity
\begin{equation}\label{eq:uV-new}
u_n(t)
=
e^{-\lambda_\ast z_n(t)}V_n(t).
\end{equation}
Moreover,
\begin{equation}\label{eq:V-eqn-new}
\partial_tV_n
=
\mathscr L_\ast V_n
+
\frac{3}{2t}V_n
-
\beta e^{-\lambda_\ast z_n(t)}V_n^2.
\end{equation}

For \(B\in\RR\), define
\begin{equation}\label{eq:Psi-new}
\Psi_n^B(t)
:=
e^{\lambda_\ast z_n(t)}
U_\ast(z_n(t)+B).
\end{equation}
Then
\[
e^{-\lambda_\ast z_n(t)}
\Psi_n^B(t)
=
U_\ast(z_n(t)+B).
\]
A direct computation using \eqref{eq:critical-wave-new} gives
\begin{equation}\label{eq:Psi-residual-new}
\begin{aligned}
&
\partial_t\Psi_n^B
-
\mathscr L_\ast\Psi_n^B
-
\frac{3}{2t}\Psi_n^B
+
\beta e^{-\lambda_\ast z_n(t)}
\bigl(\Psi_n^B\bigr)^2
\\
&\qquad=
\frac{3}{2\lambda_\ast t}
e^{\lambda_\ast z_n(t)}
U_\ast'\bigl(z_n(t)+B\bigr).
\end{aligned}
\end{equation}

\begin{lemma}[Comparison in a moving strip]
\label{lem:moving-strip-comparison-new}
Let $t_0<T_1$ and let $\ell,r:[t_0,T_1]\to\RR$ be piecewise
$C^1$ functions satisfying $\ell(t)+2\leq r(t)$. Suppose that bounded
functions $z$ and $Z$ satisfy the comparison inequality associated with
$\partial_t-\mathscr L_\ast+\mathcal Q_n(t;\cdot)$ in the interior of
\[
\{(t,n):t_0<t\leq T_1,\ \ell(t)\leq n\leq r(t)\}.
\]
If $z\leq Z$ on the initial section and on the nearest-neighbor boundary
layers at both sides of the strip, then $z\leq Z$ throughout the strip.
\end{lemma}

\begin{proof}
The proof is the discrete maximum-principle argument used in
Lemma~\ref{lem:moving-half-line-comparison-new}. The difference, after an
exponential time weight, cannot possess a positive interior maximum because
$a_\ast,b_\ast>0$ and $\mathcal Q_n(t;\cdot)$ is locally Lipschitz. The
initial and the two width-one boundary layers exclude a positive maximum on
the parabolic boundary.
\end{proof}

\begin{proposition}[Matching in a moving strip]
\label{prop:matching-new}
There exist \(B_-<B_+\), an exponent
\(\eta_0\in(0,1/4)\), and a function
\(\varepsilon(t)\to0\) such that
\begin{equation}\label{eq:matching-V-new}
\Psi_n^{B_+}(t)-\varepsilon(t)
\leq
V_n(t)
\leq
\Psi_n^{B_-}(t)+\varepsilon(t),
\end{equation}
uniformly for
\begin{equation}\label{eq:matching-strip-new}
|z_n(t)|\leq t^{\eta_0}.
\end{equation}
\end{proposition}

\begin{proof}
Let $\eta_0\in(0,1/4)$ be the exponent given by
Proposition~\ref{prop:leading-edge-new}. Fix
\[
0<\ell<\frac14-\eta_0.
\]
For $T>1$, to be chosen sufficiently large, define
\[
\Omega_T
:=
\left\{
(t,n):
t\geq T,\quad
|z_n(t)|\leq t^{\eta_0}
\right\}.
\]

On the right lateral layer of $\Omega_T$, we have
\[
z_n(t)
=
t^{\eta_0}+O(1).
\]
Since
\[
n-c_\ast t
=
z_n(t)
-
\frac{3}{2\lambda_\ast}\log t,
\]
it follows that
\[
n-c_\ast t
=
t^{\eta_0}+O(\log t).
\]
Hence, for all sufficiently large $t$,
\[
\frac12t^{\eta_0}
\leq
n-c_\ast t
\leq
2t^{\eta_0}.
\]
By Proposition~\ref{prop:leading-edge-new}, there exist constants
$0<c_1<C_1$ such that
\[
c_1t^{\eta_0}
\leq
V_n(t)
\leq
C_1t^{\eta_0}.
\]

Moreover, the normalized critical-front tail yields
\[
\begin{aligned}
\Psi_n^B(t)
&=
e^{\lambda_\ast z_n(t)}
U_\ast(z_n(t)+B)
\\
&=
e^{-\lambda_\ast B}
(z_n(t)+B)(1+o(1))
\\
&=
e^{-\lambda_\ast B}t^{\eta_0}(1+o(1)).
\end{aligned}
\]
Thus, choosing $B_+$ sufficiently large and $B_-$ sufficiently
negative, with
\[
B_-<B_+,
\]
we obtain
\[
\Psi_n^{B_+}(t)
\leq
V_n(t)
\leq
\Psi_n^{B_-}(t)
\]
on the right lateral layer for all sufficiently large $t$.

On the left lateral layer,
\[
z_n(t)
=
-t^{\eta_0}+O(1).
\]
Since
\[
V_n(t)
=
e^{\lambda_\ast z_n(t)}u_n(t),
\qquad
\Psi_n^B(t)
=
e^{\lambda_\ast z_n(t)}
U_\ast(z_n(t)+B),
\]
and
\[
0\leq u_n(t),U_\ast\leq1,
\]
we have
\[
V_n(t)+\Psi_n^B(t)
\leq
Ce^{-\lambda_\ast t^{\eta_0}}
\]
on the left lateral layer. Consequently, for every $C_2>0$ and all
sufficiently large $t$,
\[
\Psi_n^{B_+}(t)-C_2t^{-\ell}
\leq
0
\leq
V_n(t),
\]
and
\[
V_n(t)
\leq
\Psi_n^{B_-}(t)+C_2t^{-\ell}.
\]

Fix a sufficiently large time $T$. Since
\[
\left\{
n\in\ZZ:
|z_n(T)|\leq T^{\eta_0}
\right\}
\]
is finite, by increasing $B_+$, decreasing $B_-$, and then choosing
$C_2>0$ sufficiently large, we may assume that
\[
\Psi_n^{B_+}(T)-C_2T^{-\ell}
\leq
V_n(T)
\leq
\Psi_n^{B_-}(T)+C_2T^{-\ell}
\]
on the initial section of $\Omega_T$.

Set
\[
\theta_n(t)
:=
\frac{n-c_\ast t}{t^{3/8}},
\]
and define
\[
\overline s_n(t)
:=
C_2t^{-\ell}\cos\theta_n(t).
\]
Since
\[
n-c_\ast t
=
z_n(t)
-
\frac{3}{2\lambda_\ast}\log t,
\]
we have, uniformly in $\Omega_T$,
\[
|\theta_n(t)|
\leq
t^{\eta_0-3/8}
+
Ct^{-3/8}\log t
\longrightarrow0.
\]
Thus, after increasing $T$ if necessary,
\[
\overline s_n(t)>0,
\]
throughout $\Omega_T$.

A direct computation gives
\[
\begin{aligned}
(\partial_t-\mathscr L_\ast)\overline s_n
={}&
C_2t^{-\ell}
\Bigg[
\left(
s_\ast\bigl(1-\cos(t^{-3/8})\bigr)
-\frac{\ell}{t}
\right)
\cos\theta_n
\\
&\qquad+
\left(
c_\ast\bigl(t^{-3/8}-\sin(t^{-3/8})\bigr)
+
\frac{3\theta_n}{8t}
\right)
\sin\theta_n
\Bigg].
\end{aligned}
\]
It follows that
\[
\left(
\partial_t-\mathscr L_\ast-\frac{3}{2t}
\right)
\overline s_n(t)
=
\frac{s_\ast C_2}{2}
t^{-\ell-3/4}
\bigl(1+o(1)\bigr),
\]
uniformly in $\Omega_T$. In particular, there exists $c_s>0$ such
that
\[
\left(
\partial_t-\mathscr L_\ast-\frac{3}{2t}
\right)
\overline s_n(t)
\geq
c_sC_2t^{-\ell-3/4},
\]
throughout $\Omega_T$ for all sufficiently large $t$.

Define
\[
\mathcal Q_n(t;W)
:=
\beta W
-
e^{\lambda_\ast z_n(t)}
\widetilde f\left(
e^{-\lambda_\ast z_n(t)}W
\right).
\]
Then
\[
\partial_tV_n
-
\mathscr L_\ast V_n
-
\frac{3}{2t}V_n
+
\mathcal Q_n(t;V_n)
=
0.
\]
Moreover,
\[
\mathcal Q_n(t;W)
=
\begin{cases}
0,
& W\leq0,
\\[1mm]
\beta e^{-\lambda_\ast z_n(t)}W^2,
& 0\leq W\leq e^{\lambda_\ast z_n(t)},
\\[1mm]
\beta W,
& W\geq e^{\lambda_\ast z_n(t)}.
\end{cases}
\]
Hence,
\[
\mathcal Q_n(t;\cdot),
\]
is nondecreasing on $[0,\infty)$.

For $\Psi^B$, we have
\[
e^{-\lambda_\ast z_n(t)}\Psi_n^B(t)
=
U_\ast(z_n(t)+B).
\]
Since
\[
0<U_\ast<1,
\]
and
\[
\widetilde f=f
\qquad\text{on }[0,1],
\]
it follows that
\[
\mathcal Q_n(t;\Psi_n^B)
=
\beta e^{-\lambda_\ast z_n(t)}
(\Psi_n^B)^2.
\]
Define
\[
\mathscr F_n(t;W)
:=
\partial_tW_n
-
\mathscr L_\ast W_n
-
\frac{3}{2t}W_n
+
\mathcal Q_n(t;W_n).
\]
By \eqref{eq:Psi-residual-new},
\[
\mathscr F_n(t;\Psi^B)
=
\frac{3}{2\lambda_\ast t}
e^{\lambda_\ast z_n(t)}
U_\ast'(z_n(t)+B).
\]
Since $U_\ast$ is decreasing,
\[
U_\ast'(z)\leq0,
\qquad z\in\RR,
\]
and hence
\[
\mathscr F_n(t;\Psi^B)\leq0.
\]

By the derivative estimate in
Lemma~\ref{lem:critical-front-new},
\[
|U_\ast'(z)|
\leq
C(1+z)e^{-\lambda_\ast z},
\qquad z\geq0.
\]
Together with the boundedness of $U_\ast'$ on bounded intervals, this
gives
\[
\left|
\mathscr F_n(t;\Psi^B)
\right|
\leq
C_Rt^{-1+\eta_0},
\]
uniformly in $\Omega_T$, for some $C_R>0$.

Since
\[
\ell+\frac34<1-\eta_0,
\]
we may increase $T$ further so that
\[
c_sC_2t^{-\ell-3/4}
\geq
C_Rt^{-1+\eta_0}
\]
throughout $\Omega_T$.

Set
\[
\overline V_n(t)
:=
\Psi_n^{B_-}(t)+\overline s_n(t).
\]
The monotonicity of $\mathcal Q_n(t;\cdot)$ gives
\[
\mathcal Q_n(t;\overline V_n)
-
\mathcal Q_n(t;\Psi_n^{B_-})
\geq0.
\]
Therefore,
\[
\begin{aligned}
\mathscr F_n(t;\overline V)
={}&
\mathscr F_n(t;\Psi^{B_-})
+
\left(
\partial_t-\mathscr L_\ast-\frac{3}{2t}
\right)
\overline s_n
\\
&+
\mathcal Q_n(t;\overline V_n)
-
\mathcal Q_n(t;\Psi_n^{B_-})
\\
\geq{}&
-C_Rt^{-1+\eta_0}
+
c_sC_2t^{-\ell-3/4}
\geq0.
\end{aligned}
\]
Thus, $\overline V$ is a supersolution in $\Omega_T$.

Similarly, define
\[
\underline V_n(t)
:=
\Psi_n^{B_+}(t)-\overline s_n(t).
\]
At points where
\[
\underline V_n(t)>0,
\]
we have
\[
0<\underline V_n(t)\leq\Psi_n^{B_+}(t).
\]
Hence,
\[
\mathcal Q_n(t;\underline V_n)
-
\mathcal Q_n(t;\Psi_n^{B_+})
\leq0.
\]
Consequently,
\[
\begin{aligned}
\mathscr F_n(t;\underline V)
={}&
\mathscr F_n(t;\Psi^{B_+})
-
\left(
\partial_t-\mathscr L_\ast-\frac{3}{2t}
\right)
\overline s_n
\\
&+
\mathcal Q_n(t;\underline V_n)
-
\mathcal Q_n(t;\Psi_n^{B_+})
\\
\leq{}&
0.
\end{aligned}
\]
Thus, $\underline V$ is a subsolution wherever it is positive. Where
\[
\underline V_n(t)\leq0,
\]
the lower comparison is automatic because
\[
V_n(t)\geq0.
\]

The initial comparison and the comparisons on both width-one
lateral layers have been established above. Hence,
Lemma~\ref{lem:moving-strip-comparison-new} yields

\[
\underline V_n(t)
\leq
V_n(t)
\leq
\overline V_n(t)
\]
uniformly in $\Omega_T$. Since
\[
0<\overline s_n(t)\leq C_2t^{-\ell},
\]
we obtain
\[
\Psi_n^{B_+}(t)-C_2t^{-\ell}
\leq
V_n(t)
\leq
\Psi_n^{B_-}(t)+C_2t^{-\ell}
\]
whenever
\[
|z_n(t)|\leq t^{\eta_0}
\]
and $t$ is sufficiently large.

The conclusion follows with
\[
\varepsilon(t):=C_2t^{-\ell}.
\]
\end{proof}

Multiplying \eqref{eq:matching-V-new} by
\(e^{-\lambda_\ast z_n(t)}\) and using
\eqref{eq:uV-new} gives the corresponding comparison for \(u\).

\subsection{Completion of the proof}

We first control the region ahead of the matching strip.

\begin{corollary}[Decay beyond the matching strip]
\label{cor:far-ahead-new}
With \(m(t)\) defined in \eqref{eq:m-new},
\[
\sup_{n\geq m(t)+t^{\eta_0}}u_n(t)
\to0,
\qquad
t\to\infty.
\]
\end{corollary}

\begin{proof}
By Lemma~\ref{lem:upper-barrier-new},
\[
v_n(t)
\leq
\overline v_n(t),
\]
for all sufficiently large $t$ in the moving half-line
\[
y_n(t)\geq-t^\delta.
\]
Since $\xi_+$ is bounded, $w$ is bounded in
$\ell^\infty(\ZZ)$, and $p^+$ and $\Xi$ (see its definition in Lemma \ref{lem:upper-barrier-new}) are bounded, there exists
$C>0$ such that
\[
\overline v_n(t)\leq C.
\]
Hence,
\[
v_n(t)\leq C
\]
for all sufficiently large $t$ such that
\[
y_n(t)\geq-t^\delta.
\]

If
\[
n\geq m(t)+t^{\eta_0},
\]
then
\[
\begin{aligned}
y_n(t)
&=
n-c_\ast t
\\
&\geq
t^{\eta_0}
-
\frac{3}{2\lambda_\ast}\log t.
\end{aligned}
\]
Since $\eta_0>0$, for all sufficiently large $t$,
\[
t^{\eta_0}
-
\frac{3}{2\lambda_\ast}\log t
\geq
\frac12t^{\eta_0}.
\]
Therefore,
\[
y_n(t)\geq\frac12t^{\eta_0},
\]
and, in particular,
\[
y_n(t)\geq-t^\delta.
\]
Thus the preceding upper bound for $v_n(t)$ applies. Using
\[
u_n(t)
=
e^{-\lambda_\ast y_n(t)}v_n(t),
\]
we obtain
\[
u_n(t)
\leq
Ce^{-\lambda_\ast y_n(t)}
\leq
Ce^{-\lambda_\ast t^{\eta_0}/2}.
\]
Taking the supremum over
\[
n\geq m(t)+t^{\eta_0},
\]
gives
\[
\sup_{n\geq m(t)+t^{\eta_0}}u_n(t)
\leq
Ce^{-\lambda_\ast t^{\eta_0}/2}
\longrightarrow0.
\]
\end{proof}

\begin{proposition}[Completion of the proof of Theorem~\ref{Thm:Bramson}]
Let $\theta\in(0,1)$ and let
\begin{equation*}
m(t)=c_*t-\frac{3}{2\lambda_*}\log t.
\end{equation*}
Then there exist constants $C_\theta>0$ and $T_\theta>0$ such that, for every $t\geq T_\theta$,
\begin{equation*}
m(t)-C_\theta
\leq \kappa_\theta(t)
\leq m(t)+C_\theta.
\end{equation*}
Consequently,
\begin{equation*}
\kappa_\theta(t)=c_*t-\frac{3}{2\lambda_*}\log t+O(1),
\qquad t\to\infty,
\end{equation*}
where the constant implicit in $O(1)$ may depend on $\theta$.
\end{proposition}

\begin{proof}
Fix $\theta\in(0,1)$. By Proposition~\ref{prop:matching-new}, there exist $B_-<B_+$ and a function $\varepsilon(t)\to0$ such that
\begin{equation*}
\Psi_n^{B_+}(t)-\varepsilon(t)\leq V_n(t)\leq\Psi_n^{B_-}(t)+\varepsilon(t)
\end{equation*}
whenever $|z_n(t)|\leq t^{\eta_0}$, where
\begin{equation*}
z_n(t)=n-m(t).
\end{equation*}
Since $u_n(t)=e^{-\lambda_*z_n(t)}V_n(t)$ and
\begin{equation*}
e^{-\lambda_*z_n(t)}\Psi_n^B(t)=U_*(z_n(t)+B),
\end{equation*}
we obtain
\begin{align*}
U_*(z_n(t)+B_+)-e^{-\lambda_*z_n(t)}\varepsilon(t)
&\leq u_n(t) \\
&\leq U_*(z_n(t)+B_-)+e^{-\lambda_*z_n(t)}\varepsilon(t).
\end{align*}

Choose $\delta>0$ such that
\begin{equation*}
0<2\delta<\min\{\theta,1-\theta\}.
\end{equation*}
Since $U_*$ is strictly decreasing, with $U_*(-\infty)=1$ and $U_*(+\infty)=0$, there exists $R_\theta>0$ such that
\begin{equation*}
U_*(-R_\theta+B_+)>\theta+2\delta,
\qquad
U_*(R_\theta+B_-)<\theta-2\delta.
\end{equation*}

We now make the time choices explicit. Since $R_\theta$ is fixed and $\eta_0>0$, there exists $T_0=T_0(R_\theta)$ such that
\begin{equation*}
R_\theta+1\leq t^{\eta_0},
\qquad t\geq T_0.
\end{equation*}
Moreover, since $\varepsilon(t)\to0$, there exists $T_1=T_1(\theta)$ such that
\begin{equation*}
\varepsilon(t)<\delta e^{-\lambda_*(R_\theta+1)},
\qquad t\geq T_1.
\end{equation*}
Let $n_t^-:=\lfloor m(t)-R_\theta\rfloor,$ where $\lfloor \cdot\rfloor$  denotes the greatest integer function. Then
\begin{equation*}
z_{n_t^-}(t)\in(-R_\theta-1,-R_\theta].
\end{equation*}
Thus, for $t\geq T_0$, the matching estimate applies at $n_t^-$. Recalling that $\lambda_*>0$, the preceding inclusion yields
\begin{equation*}
e^{-\lambda_*z_{n_t^-}(t)}\varepsilon(t)
\leq e^{\lambda_*(R_\theta+1)}\varepsilon(t).
\end{equation*}
Hence, for $t\geq T_1$,
\begin{align*}
u_{n_t^-}(t)
&\geq U_*(-R_\theta+B_+)-e^{\lambda_*(R_\theta+1)}\varepsilon(t)\\
&>\theta+2\delta-\delta\\
&>\theta.
\end{align*}
Consequently,
\begin{equation*}
\kappa_\theta(t)\geq n_t^-\geq m(t)-R_\theta-1.
\end{equation*}

For the upper bound, consider first the indices satisfying
\begin{equation*}
R_\theta\leq z_n(t)\leq t^{\eta_0}.
\end{equation*}
The upper comparison and the monotonicity of $U_*$ give
\begin{equation*}
u_n(t)\leq U_*(R_\theta+B_-)+e^{-\lambda_*R_\theta}\varepsilon(t).
\end{equation*}
Since $\varepsilon(t)\to0$, there exists $T_2=T_2(\theta)$ such that
\begin{equation*}
e^{-\lambda_*R_\theta}\varepsilon(t)<\delta,
\qquad t\geq T_2.
\end{equation*}
Hence, for $t\geq T_2$,
\begin{equation*}
u_n(t)<\theta-2\delta+\delta<\theta
\end{equation*}
whenever $R_\theta\leq z_n(t)\leq t^{\eta_0}$.

Finally, Corollary~\ref{cor:far-ahead-new} implies that there exists $T_3=T_3(\theta)$ such that
\begin{equation*}
\sup_{n\geq m(t)+t^{\eta_0}}u_n(t)<\theta,
\qquad t\geq T_3.
\end{equation*}
Taking
\begin{equation*}
T_\theta:=\max\{T_0,T_1,T_2,T_3\},
\end{equation*}
we conclude that $u_n(t)<\theta$ for every $n\geq m(t)+R_\theta$ and every $t\geq T_\theta$. Therefore,
\begin{equation*}
\kappa_\theta(t)\leq m(t)+R_\theta+1.
\end{equation*}
Combining the lower and upper bounds, we have
\begin{equation*}
m(t)-R_\theta-1
\leq \kappa_\theta(t)
\leq m(t)+R_\theta+1.
\end{equation*}
Setting $C_\theta:=R_\theta+1$, we obtain
\begin{equation*}
m(t)-C_\theta
\leq \kappa_\theta(t)
\leq m(t)+C_\theta.
\end{equation*}
This proves the claim.

Combining both bounds proves the claim. Notice that the choices have the dependence
\begin{equation*}
\theta\longmapsto \delta, R_\theta\longmapsto T_\theta.
\end{equation*}
In particular, the bounded remainder in the conclusion is allowed to depend on $\theta$, as stated in Theorem~1.2.
\end{proof}

\begin{remark}
The branching parameter \(q\) changes the dispersion relation and,
therefore, both \(c_\ast\) and \(\lambda_\ast\). After critical
conjugation, however, the leading edge is governed by a one-dimensional
nearest-neighbor operator. The odd-data cancellation produces the
dipole scale \(t^{-3/2}\), while the double root of the dispersion
relation produces the critical tail
\[
z e^{-\lambda_\ast z}.
\]
Matching these two scales gives the logarithmic coefficient
\[
\frac{3}{2\lambda_\ast}.
\]
\end{remark}

\appendix

\section{Auxiliary Estimates for the Linear Problem}\label{Appendix:Auxiliary:Results}

\begin{lemma}\label{Lemma:radial-linear-bound}
Let \(q\geq2\) and \(\alpha,\beta>0\). Let
\(h(t)=\{h_n(t)\}_{n\in\NN_0}\) be the solution of
\[
\begin{rcases}
\partial_t h_0(t)
=
\alpha(q+1)h_1(t)
+\bigl(\beta-\alpha(q+1)\bigr)h_0(t),
\\[1mm]
\partial_t h_n(t)
=
\alpha h_{n-1}(t)
+\alpha q h_{n+1}(t)
+\bigl(\beta-\alpha(q+1)\bigr)h_n(t),
\qquad n\geq1,
\\[1mm]
h_n(0)=\delta_{n0},
\qquad n\in\NN_0.
\end{rcases}
\]
Then
\[
0\leq h_n(t)\leq e^{\beta t},
\qquad
t\geq0,\quad n\in\NN_0.
\]
\end{lemma}

\begin{proof}
Define the bounded operator \(A\) on \(\ell^\infty(\NN_0)\) by
\[
(Az)_0
:=
\alpha(q+1)(z_1-z_0),
\]
and, for \(n\geq1\),
\[
(Az)_n
:=
\alpha z_{n-1}
+\alpha q z_{n+1}
-\alpha(q+1)z_n.
\]
Then the system can be written as
\[
\partial_t h=(A+\beta I)h,
\qquad
h(0)=\delta_0,
\]
and hence
\[
h(t)=e^{t(A+\beta I)}\delta_0.
\]

We first observe that the semigroup generated by \(A\) is positive.
Indeed, setting
\[
B:=A+\alpha(q+1)I,
\]
we have
\[
(Bz)_0=\alpha(q+1)z_1,\quad \text{and} \quad  (Bz)_n
=
\alpha z_{n-1}+\alpha qz_{n+1},\quad n\geq1
\]
Clearly \(B\) is a positive operator, and
\[
e^{tA}
=
e^{-\alpha(q+1)t}e^{tB}
=
e^{-\alpha(q+1)t}
\sum_{k=0}^\infty\frac{t^k}{k!}B^k
\]
is positive as well. Consequently,
\(e^{t(A+\beta I)}=e^{\beta t}e^{tA}\) is a positive semigroup.

Now,
\[
0\leq \delta_{n0}\leq1,
\qquad n\in\NN_0.
\]
Hence, coordinate by coordinate,
\[
0
\leq
\bigl(e^{t(A+\beta I)}\delta_0\bigr)_n
\leq
\bigl(e^{t(A+\beta I)}\mathbf 1\bigr)_n,
\qquad n\in\NN_0,
\]
where
\[
\mathbf 1=(1,1,\ldots).
\]
Since
\[
A\mathbf 1=0,
\]
we obtain
\[
e^{t(A+\beta I)}\mathbf 1
=
e^{\beta t}\mathbf 1.
\]
Therefore, for every \(n\in\NN_0\),
\[
0
\leq
h_n(t)
=
\bigl(e^{t(A+\beta I)}\delta_0\bigr)_n
\leq
e^{\beta t}.
\]
This proves the result.
\end{proof}

\begin{lemma}\label{Lemma:Solution:Recurrence}
Let \(q\geq2\) and \(\alpha,\beta>0\). For each fixed
\(\lambda>\beta\), consider the recurrence relation
\begin{equation}
\label{eq:Laplace-bulk-unnormalized-appendix}
qH_{n+1}(\lambda)
-\frac{\lambda-\beta+\alpha(q+1)}{\alpha}H_n(\lambda)
+H_{n-1}(\lambda)
=0,
\qquad n\geq1.
\end{equation}
Then its general solution is given by
\[
H_n(\lambda)
=
A_-(\lambda)r_-(\lambda)^n
+
A_+(\lambda)r_+(\lambda)^n,
\]
where
\[
r_{\pm}(\lambda)
=
\frac{
\lambda-\beta+\alpha(q+1)
\pm
\sqrt{\bigl(\lambda-\beta+\alpha(q+1)\bigr)^2
-4\alpha^2q}
}
{2\alpha q}.
\]
Moreover,
\[
0<r_-(\lambda)<1<r_+(\lambda).
\]
In particular, if
\[
\sup_{n\geq0}|H_n(\lambda)|<\infty,
\]
then necessarily
\[
A_+(\lambda)=0,
\]
and hence
\[
H_n(\lambda)
=
H_0(\lambda)r_-(\lambda)^n.
\]
\end{lemma}

\begin{proof}
For fixed \(\lambda>\beta\), the characteristic polynomial associated with
\eqref{eq:Laplace-bulk-unnormalized-appendix} is
\[
P_\lambda(r)
=
qr^2
-\frac{\lambda-\beta+\alpha(q+1)}{\alpha}r
+1.
\]
Its discriminant is
\[
\Delta_\lambda
=
\left(
\frac{\lambda-\beta+\alpha(q+1)}{\alpha}
\right)^2
-4q.
\]
Since \(\lambda>\beta\),
\[
\frac{\lambda-\beta+\alpha(q+1)}{\alpha}
>
q+1,
\]
and therefore
\[
\Delta_\lambda
>
(q+1)^2-4q
=
(q-1)^2
>0.
\]
Hence \(P_\lambda\) has two distinct real roots, namely
\[
r_{\pm}(\lambda)
=
\frac{
\lambda-\beta+\alpha(q+1)
\pm
\sqrt{\bigl(\lambda-\beta+\alpha(q+1)\bigr)^2
-4\alpha^2q}
}
{2\alpha q}.
\]
It follows that the general solution of
\eqref{eq:Laplace-bulk-unnormalized-appendix} is
\[
H_n(\lambda)
=
A_-(\lambda)r_-(\lambda)^n
+
A_+(\lambda)r_+(\lambda)^n.
\]

We next determine the location of the roots. By Vieta's formulas,
\[
r_-(\lambda)r_+(\lambda)
=
\frac1q
>0,
\]
and
\[
r_-(\lambda)+r_+(\lambda)
=
\frac{\lambda-\beta+\alpha(q+1)}{\alpha q}
>0.
\]
Thus both roots are positive.

Moreover,
\[
P_\lambda(0)=1>0,
\]
whereas
\[
P_\lambda(1)
=
q+1
-
\frac{\lambda-\beta+\alpha(q+1)}{\alpha}
=
-\frac{\lambda-\beta}{\alpha}
<0.
\]
Therefore, by the Intermediate Value Theorem, \(P_\lambda\) has a root
in \((0,1)\). Since
\[
r_-(\lambda)<r_+(\lambda),
\]
this root must be \(r_-(\lambda)\). Hence
\[
0<r_-(\lambda)<1.
\]

On the other hand,
\[
P_\lambda(r)\to+\infty
\qquad\text{as }r\to+\infty.
\]
Since \(P_\lambda(1)<0\), the Intermediate Value Theorem implies that
\(P_\lambda\) has a root in \((1,\infty)\). This root must be
\(r_+(\lambda)\), and therefore
\[
r_+(\lambda)>1.
\]
Consequently,
\[
0<r_-(\lambda)<1<r_+(\lambda).
\]

Finally, suppose that
\[
\sup_{n\geq0}|H_n(\lambda)|<\infty.
\]
Since \(r_+(\lambda)>1\), the term
\[
A_+(\lambda)r_+(\lambda)^n,
\]
is unbounded as \(n\to\infty\) unless
\[
A_+(\lambda)=0.
\]
Hence
\[
H_n(\lambda)
=
A_-(\lambda)r_-(\lambda)^n.
\]
Evaluating at \(n=0\), we obtain
\[
A_-(\lambda)=H_0(\lambda),
\]
and therefore
\[
H_n(\lambda)
=
H_0(\lambda)r_-(\lambda)^n.
\]
\end{proof}

\begin{lemma}\label{Lemma:Laplace-Bessel-over-t}
Let \(a>0\), \(s>a\), and \(m\in\NN\). Then
\begin{equation}
\label{eq:Laplace-Bessel-over-t}
\int_0^\infty
\frac{I_m(at)}{t}e^{-st}\,dt
=
\frac{1}{m}
\left(
\frac{s-\sqrt{s^2-a^2}}{a}
\right)^m.
\end{equation}
\end{lemma}

\begin{proof}
Recall the classical Laplace transform formula
\[
\int_0^\infty e^{-st}I_m(at)\,dt
=
\frac{1}{\sqrt{s^2-a^2}}
\left(
\frac{s-\sqrt{s^2-a^2}}{a}
\right)^m,
\qquad s>a,
\]
see \cite[Formula~6.611.1]{Gradshteyn-Ryzhik-2015}.

We also use the standard identity
\[
\mathcal L\left[\frac{f(t)}{t}\right](s)
=
\int_s^\infty \mathcal L[f](\tau)\,d\tau.
\]
Applying it with \(f(t)=I_m(at)\), we obtain
\[
\int_0^\infty
\frac{I_m(at)}{t}e^{-st}\,dt
=
\int_s^\infty
\frac{1}{\sqrt{\tau^2-a^2}}
\left(
\frac{\tau-\sqrt{\tau^2-a^2}}{a}
\right)^m
d\tau.
\]

Set
\[
\omega
=
\frac{\tau-\sqrt{\tau^2-a^2}}{a}.
\]
Then
\[
\tau
=
\frac{a}{2}\left(\omega+\frac{1}{\omega}\right),
\]
and
\[
\sqrt{\tau^2-a^2}
=
\frac{a}{2}\left(\frac{1}{\omega}-\omega\right).
\]
Differentiating,
\[
d\tau
=
\frac{a}{2}\left(1-\frac{1}{\omega^2}\right)d\omega,
\]
and therefore
\[
\frac{d\tau}{\sqrt{\tau^2-a^2}}
=
-\frac{d\omega}{\omega}.
\]

When \(\tau=s\), we have
\[
\omega
=
\frac{s-\sqrt{s^2-a^2}}{a},
\]
whereas
\[
\omega\to0
\quad\text{as }\tau\to\infty.
\]
Hence,
\[
\begin{aligned}
\int_0^\infty
\frac{I_m(at)}{t}e^{-st}\,dt
&=
\int_0^{\frac{s-\sqrt{s^2-a^2}}{a}}
\omega^{m-1}\,d\omega=
\frac{1}{m}
\left(
\frac{s-\sqrt{s^2-a^2}}{a}
\right)^m.
\end{aligned}
\]
This proves \eqref{eq:Laplace-Bessel-over-t}.
\end{proof}

\begin{lemma}\label{Lemma:Bessel-series-comparison}
Let \(q\geq2\), \(n\in\NN_0\), and \(z>0\). Then
\[
(n+1)I_{n+1}(z)
\leq
\sum_{j=0}^{\infty}
q^{-j}(n+2j+1)I_{n+2j+1}(z)
\leq
\frac{q(q+1)}{(q-1)^2}
(n+1)I_{n+1}(z).
\]
\end{lemma}

\begin{proof}
The classical ratio estimate
\[
0<\frac{I_{\nu+1}(z)}{I_\nu(z)}<1,
\qquad z>0,\quad \nu\geq0,
\]
see, for instance, Amos \cite{Amos-1974}, implies that for every fixed
\(z>0\) the sequence \(m\mapsto I_m(z)\) is decreasing for
\(m\in\NN_0\). Therefore,
\[
I_{n+2j+1}(z)
\leq
I_{n+1}(z),
\qquad j\geq0.
\]
It follows that
\[
\begin{aligned}
\sum_{j=0}^{\infty}
q^{-j}(n+2j+1)I_{n+2j+1}(z)
&\leq
I_{n+1}(z)
\sum_{j=0}^{\infty}
q^{-j}(n+2j+1).
\end{aligned}
\]
Using
\[
\sum_{j=0}^{\infty}q^{-j}
=
\frac{q}{q-1},
\]
and
\[
\sum_{j=0}^{\infty}jq^{-j}
=
\frac{q}{(q-1)^2},
\]
we obtain
\[
\begin{aligned}
\sum_{j=0}^{\infty}
q^{-j}(n+2j+1)
&=
(n+1)\frac{q}{q-1}
+
2\frac{q}{(q-1)^2}
\\
&=
\frac{
q\bigl((q-1)n+q+1\bigr)
}{
(q-1)^2
}.
\end{aligned}
\]
Since
\[
(q-1)n+q+1
\leq
(q+1)(n+1),
\]
we conclude that
\[
\sum_{j=0}^{\infty}
q^{-j}(n+2j+1)I_{n+2j+1}(z)
\leq
\frac{q(q+1)}{(q-1)^2}
(n+1)I_{n+1}(z).
\]

For the lower bound, it is enough to retain the term \(j=0\). Since
\(I_m(z)>0\) for \(m\in\NN_0\) and \(z>0\),
\[
\sum_{j=0}^{\infty}
q^{-j}(n+2j+1)I_{n+2j+1}(z)
\geq
(n+1)I_{n+1}(z).
\]
This proves the result.
\end{proof}

\begin{lemma}\label{Lemma:Bessel-ratio-uniform}
For every \(n\in\NN_0\) and \(z>0\),
\[
\frac{1}{2+\sqrt2}
\frac{z}{\left(1+n^2+z^2\right)^{1/2}}
\leq
\frac{I_{n+1}(z)}{I_n(z)}
\leq
2
\frac{z}{\left(1+n^2+z^2\right)^{1/2}}.
\]
\end{lemma}

\begin{proof}
The Amos bounds \cite{Amos-1974} give
\[
\frac{z}
{n+\frac12+\sqrt{z^2+\left(n+\frac32\right)^2}}
\leq
\frac{I_{n+1}(z)}{I_n(z)}
\leq
\frac{z}
{n+\frac12+\sqrt{z^2+\left(n+\frac12\right)^2}}.
\]

We first estimate the denominator in the lower bound. Since
\[
n+\frac12
\leq
\sqrt2\left(1+n^2+z^2\right)^{1/2},
\]
and
\[
\sqrt{z^2+\left(n+\frac32\right)^2}
\leq
2\left(1+n^2+z^2\right)^{1/2},
\]
we obtain
\[
n+\frac12
+
\sqrt{z^2+\left(n+\frac32\right)^2}
\leq
(2+\sqrt2)
\left(1+n^2+z^2\right)^{1/2}.
\]
Hence,
\[
\frac{I_{n+1}(z)}{I_n(z)}
\geq
\frac{1}{2+\sqrt2}
\frac{z}{\left(1+n^2+z^2\right)^{1/2}}.
\]

For the upper bound, observe that
\[
\sqrt{z^2+\left(n+\frac12\right)^2}
\geq
\frac12
\left(1+n^2+z^2\right)^{1/2}.
\]
Therefore,
\[
n+\frac12
+
\sqrt{z^2+\left(n+\frac12\right)^2}
\geq
\frac12
\left(1+n^2+z^2\right)^{1/2},
\]
which gives
\[
\frac{I_{n+1}(z)}{I_n(z)}
\leq
2
\frac{z}{\left(1+n^2+z^2\right)^{1/2}}.
\]
This completes the proof.
\end{proof}

\begin{lemma}\label{Lemma:Bessel-uniform-estimate}
There exist absolute constants \(c_B,C_B>0\) such that, for every
\(n\in\NN_0\) and \(z>0\),
\[
c_B
\frac{
\exp\left\{
\sqrt{n^2+z^2}
+
n\log\left(
\frac{z}{n+\sqrt{n^2+z^2}}
\right)
\right\}
}{
\left(1+n^2+z^2\right)^{1/4}
}
\leq
I_n(z),
\]
and
\[
I_n(z)
\leq
C_B
\frac{
\exp\left\{
\sqrt{n^2+z^2}
+
n\log\left(
\frac{z}{n+\sqrt{n^2+z^2}}
\right)
\right\}
}{
\left(1+n^2+z^2\right)^{1/4}
}.
\]
For \(n=0\), the logarithmic term is understood to be zero.
\end{lemma}

\begin{proof}
Recall that the heat kernel on \(\mathbb Z\) is given by
\[
h_z^{\mathbb Z}(n)
=
e^{-z}I_n(z).
\]
By \cite[Theorem~2.3]{Cowling-Meda-Setti-2000-TAMS}, there exist
absolute constants \(c_B,C_B>0\) such that, for every \(n\geq1\) and
\(z>0\),
\[
c_B
\frac{
\exp\{-z+n\xi(z/n)\}
}{
(1+n^2+z^2)^{1/4}
}
\leq
h_z^{\mathbb Z}(n)
\leq
C_B
\frac{
\exp\{-z+n\xi(z/n)\}
}{
(1+n^2+z^2)^{1/4}
},
\]
where
\[
\xi(s)
=
\sqrt{1+s^2}
+
\log\left(
\frac{s}{1+\sqrt{1+s^2}}
\right),
\qquad s>0.
\]
Since
\[
I_n(z)=e^z h_z^{\mathbb Z}(n),
\]
we obtain
\[
c_B
\frac{
\exp\{n\xi(z/n)\}
}{
(1+n^2+z^2)^{1/4}
}
\leq
I_n(z)
\leq
C_B
\frac{
\exp\{n\xi(z/n)\}
}{
(1+n^2+z^2)^{1/4}
}.
\]

For \(n\geq1\), a direct computation gives
\[
\begin{aligned}
n\xi(z/n)
&=
n\sqrt{1+\frac{z^2}{n^2}}
+
n\log\left(
\frac{z/n}{1+\sqrt{1+z^2/n^2}}
\right)
\\
&=
\sqrt{n^2+z^2}
+
n\log\left(
\frac{z}{n+\sqrt{n^2+z^2}}
\right).
\end{aligned}
\]
This proves the desired estimate for \(n\geq1\).

It remains to consider \(n=0\). The standard estimate for the modified
Bessel function \(I_0\) gives
\[
c_0\frac{e^z}{(1+z)^{1/2}}
\leq
I_0(z)
\leq
C_0\frac{e^z}{(1+z)^{1/2}},
\qquad z>0,
\]
for suitable absolute constants \(c_0,C_0>0\). Moreover, there exist
absolute constants \(c,C>0\) such that
\[
c(1+z^2)^{1/4}
\leq
(1+z)^{1/2}
\leq
C(1+z^2)^{1/4}.
\]
Hence, after possibly decreasing \(c_B\) and increasing \(C_B\),
\[
c_B
\frac{e^z}{(1+z^2)^{1/4}}
\leq
I_0(z)
\leq
C_B
\frac{e^z}{(1+z^2)^{1/4}}.
\]
Since
\[
\sqrt{0^2+z^2}=z,
\]
this is precisely the claimed estimate when \(n=0\).
\end{proof}

\begin{corollary}\label{Cor:Bessel-nplusone-estimate}
Let \(\rho>0\). There exist constants \(c_\rho,C_\rho>0\) such that,
for every \(n\in\NN_0\) and \(t>0\),
\[
c_\rho
\frac{
\exp\left\{
\sqrt{n^2+\rho^2t^2}
+
n\log\left(
\frac{\rho t}
{n+\sqrt{n^2+\rho^2t^2}}
\right)
\right\}
}{
\left(1+n^2+\rho^2t^2\right)^{3/4}
}
\leq
\frac{1}{t}I_{n+1}(\rho t)
\]
and
\[
\frac{1}{t}I_{n+1}(\rho t)
\leq
C_\rho
\frac{
\exp\left\{
\sqrt{n^2+\rho^2t^2}
+
n\log\left(
\frac{\rho t}
{n+\sqrt{n^2+\rho^2t^2}}
\right)
\right\}
}{
\left(1+n^2+\rho^2t^2\right)^{3/4}
}.
\]
More precisely, one may take
\[
c_\rho
=
\frac{\rho c_B}{2+\sqrt2},
\qquad
C_\rho
=
2\rho C_B.
\]
\end{corollary}

\begin{proof}
Set
\[
z:=\rho t.
\]
Then
\[
\frac{1}{t}I_{n+1}(z)
=
\frac{\rho}{z}
\frac{I_{n+1}(z)}{I_n(z)}
I_n(z).
\]
By Lemma~\ref{Lemma:Bessel-ratio-uniform},
\[
\frac{\rho}{2+\sqrt2}
\frac{I_n(z)}
{(1+n^2+z^2)^{1/2}}
\leq
\frac{1}{t}I_{n+1}(z)
\leq
2\rho
\frac{I_n(z)}
{(1+n^2+z^2)^{1/2}}.
\]
Applying Lemma~\ref{Lemma:Bessel-uniform-estimate} to \(I_n(z)\), we
obtain
\[
\frac{\rho c_B}{2+\sqrt2}
\frac{
\exp\left\{
\sqrt{n^2+z^2}
+
n\log\left(
\frac{z}{n+\sqrt{n^2+z^2}}
\right)
\right\}
}{
(1+n^2+z^2)^{3/4}
}
\leq
\frac1t I_{n+1}(z),
\]
and
\[
\frac1t I_{n+1}(z)
\leq
2\rho C_B
\frac{
\exp\left\{
\sqrt{n^2+z^2}
+
n\log\left(
\frac{z}{n+\sqrt{n^2+z^2}}
\right)
\right\}
}{
(1+n^2+z^2)^{3/4}
}.
\]
Substituting \(z=\rho t\) gives the result.
\end{proof}

The following result provides the Taylor expansion of the function
\(\Phi\) near the critical speed \(c_\ast\).

\begin{lemma}\label{lem:Phi-expansion-critical}
Let
\[
\Phi(c)
=
\beta-\alpha(q+1)
+\sqrt{c^2+4q\alpha^2}
+
c\log\left(
\frac{2\alpha}
{c+\sqrt{c^2+4q\alpha^2}}
\right),
\qquad c\geq0,
\]
and let \(c_\ast>0\) be the critical speed. Set
\[
s_\ast
:=
\sqrt{c_\ast^2+4q\alpha^2}.
\]
If \(\xi=\xi(t)\) satisfies
\[
\xi=o(t)
\qquad\text{as }t\to\infty,
\]
then
\[
t\Phi\left(c_\ast+\frac{\xi}{t}\right)
=
-\lambda_\ast\xi
-\frac{\xi^2}{2s_\ast t}
+
O\left(\frac{|\xi|^3}{t^2}\right).
\]
In particular, if
\[
|\xi|=o(t^{2/3}),
\]
then
\[
t\Phi\left(c_\ast+\frac{\xi}{t}\right)
=
-\lambda_\ast\xi
-\frac{\xi^2}{2s_\ast t}
+o(1).
\]
\end{lemma}

\begin{proof}
A direct differentiation gives
\[
\Phi'(c)
=
\log\left(
\frac{2\alpha}
{c+\sqrt{c^2+4q\alpha^2}}
\right),
\]
and therefore
\[
\begin{aligned}
\Phi''(c)
&=
-\frac{d}{dc}
\log\left(
c+\sqrt{c^2+4q\alpha^2}
\right)
\\
&=
-\frac{
1+\dfrac{c}{\sqrt{c^2+4q\alpha^2}}
}{
c+\sqrt{c^2+4q\alpha^2}
}
\\
&=
-\frac{1}{\sqrt{c^2+4q\alpha^2}}.
\end{aligned}
\]
In particular,
\[
\Phi''(c_\ast)
=
-\frac1{s_\ast}.
\]
Moreover,
\[
\Phi'''(c)
=
\frac{c}
{\left(c^2+4q\alpha^2\right)^{3/2}},
\]
so \(\Phi'''\) is bounded in a neighborhood of \(c_\ast\).

Since \(\xi=o(t)\), one has
\[
c_\ast+\frac{\xi}{t}\longrightarrow c_\ast.
\]
Taylor's formula at \(c_\ast\) therefore yields
\[
\Phi\left(c_\ast+\frac{\xi}{t}\right)
=
\Phi(c_\ast)
+
\Phi'(c_\ast)\frac{\xi}{t}
+
\frac12\Phi''(c_\ast)\frac{\xi^2}{t^2}
+
O\left(\frac{|\xi|^3}{t^3}\right).
\]
By the definition of the critical pair,
\[
\Phi(c_\ast)=0
\]
and
\[
\lambda_\ast
=
-\Phi'(c_\ast)
=
\log\left(
\frac{
c_\ast+\sqrt{c_\ast^2+4q\alpha^2}
}{
2\alpha
}
\right).
\]
Hence
\[
\Phi'(c_\ast)=-\lambda_\ast,
\qquad
\Phi''(c_\ast)=-\frac1{s_\ast}.
\]
Multiplying the Taylor expansion by \(t\), we obtain
\[
t\Phi\left(c_\ast+\frac{\xi}{t}\right)
=
-\lambda_\ast\xi
-\frac{\xi^2}{2s_\ast t}
+
O\left(\frac{|\xi|^3}{t^2}\right).
\]

Finally, if
\[
|\xi|=o(t^{2/3}),
\]
then
\[
\frac{|\xi|^3}{t^2}=o(1),
\]
which gives the last assertion.
\end{proof}

\section{Basic Properties of the Nonlinear Problem}\label{Appendix:Auxiliary:Results:2}
\begin{lemma}\label{lem:nonlinear-wellposedness-radiality}
Let \(u_0\in\ell^\infty(\TT_{q+1})\) satisfy
\[
0\leq u_0(x)\leq1,
\qquad x\in\TT_{q+1}.
\]
Then problem~\eqref{Eq:Fisher-KPP-tree:nonlinear} admits a unique global
solution
\[
u\in C^1\bigl([0,\infty);\ell^\infty(\TT_{q+1})\bigr).
\]
Moreover, the comparison principle holds and
\[
0\leq u(t,x)\leq1,
\qquad
t\geq0,\quad x\in\TT_{q+1}.
\]
If, in addition, \(u_0\) is radial with respect to the root
\(o\in\TT_{q+1}\), then the solution remains radial for all \(t\geq0\).
In this case, we may write
\[
u(t,x)=u_n(t),
\qquad
n=d(o,x).
\]
\end{lemma}

\begin{proof}
We regard \eqref{Eq:Fisher-KPP-tree:nonlinear} as an ordinary
differential equation on the Banach space
\(\ell^\infty(\TT_{q+1})\). Since every vertex of
\(\TT_{q+1}\) has degree \(q+1\),
\[
\|\Delta_{\TT_{q+1}}z\|_{\ell^\infty}
\leq
2(q+1)\|z\|_{\ell^\infty},
\]
so that \(\Delta_{\TT_{q+1}}\) is a bounded operator on
\(\ell^\infty(\TT_{q+1})\).

Let \(\widetilde f:\mathbb R\to\mathbb R\) be a globally Lipschitz
function such that
\[
\widetilde f(s)=\beta s(1-s),
\qquad 0\leq s\leq1.
\]
Then
\[
\mathcal A(z)
:=
\alpha\Delta_{\TT_{q+1}}z+\widetilde f(z)
\]
defines a globally Lipschitz map on
\(\ell^\infty(\TT_{q+1})\). Hence, by the standard existence and
uniqueness theorem for ordinary differential equations in Banach spaces
\cite[Chapter~1]{Deimling-1977-Springer}, the extended problem
\[
u'(t)=\mathcal A(u(t)),
\qquad
u(0)=u_0,
\]
admits a unique global solution
\[
u\in C^1\bigl([0,\infty);\ell^\infty(\TT_{q+1})\bigr).
\]
We next establish the comparison principle. Let \(u\) and \(v\) be two
solutions of the extended equation such that
\[
u(0,\cdot)\leq v(0,\cdot),
\]
and set \(w:=u-v\). Then, for every \(x\in\TT_{q+1}\),
\[
\partial_t w(t,x)
=
\alpha\sum_{y\sim x}w(t,y)
-\alpha(q+1)w(t,x)
+\widetilde f(u(t,x))-\widetilde f(v(t,x)).
\]

Let \(L>0\) be a Lipschitz constant for \(\widetilde f\), and define
\[
M(t):=\sup_{x\in\TT_{q+1}}w(t,x)^+,
\qquad
w(t,x)^+:=\max\{w(t,x),0\}.
\]
For each fixed \(t\) such that \(M(t)>0\), choose a sequence
\((x_k)\subset\TT_{q+1}\) satisfying
\[
w(t,x_k)>0,
\qquad
w(t,x_k)\to M(t).
\]
Since
\[
w(t,y)\leq M(t),
\qquad y\in\TT_{q+1},
\]
we obtain
\[
\alpha\sum_{y\sim x_k}w(t,y)
-\alpha(q+1)w(t,x_k)
\leq
\alpha(q+1)\bigl(M(t)-w(t,x_k)\bigr).
\]
Moreover, since \(w(t,x_k)>0\), the Lipschitz continuity of
\(\widetilde f\) gives
\[
\widetilde f(u(t,x_k))-\widetilde f(v(t,x_k))
\leq
Lw(t,x_k).
\]
Hence
\[
\partial_t w(t,x_k)
\leq
\alpha(q+1)\bigl(M(t)-w(t,x_k)\bigr)
+
Lw(t,x_k).
\]
Passing to the limit along the maximizing sequence yields
\[
D^+M(t)\leq LM(t),
\]
where \(D^+\) denotes the upper Dini derivative. Since
\(u(0,\cdot)\leq v(0,\cdot)\), we have \(M(0)=0\), and Gronwall's
inequality gives
\[
M(t)=0,
\qquad t\geq0.
\]
Therefore,
\[
w(t,x)\leq0,
\qquad
t\geq0,\quad x\in\TT_{q+1},
\]
and hence
\[
u(t,x)\leq v(t,x),
\qquad
t\geq0,\quad x\in\TT_{q+1}.
\]

We apply this principle to the constant functions
\[
\underline u(t,x)\equiv0,
\qquad
\overline u(t,x)\equiv1.
\]
Since
\[
\widetilde f(0)=\widetilde f(1)=0,
\]
and the graph Laplacian annihilates constant functions,
\(\underline u\) and \(\overline u\) are solutions of the extended
equation. Thus, from \(0\leq u_0\leq1\), comparison gives
\[
0\leq u(t,x)\leq1,
\qquad
t\geq0,\quad x\in\TT_{q+1}.
\]
Consequently, the solution remains in the interval on which
\(\widetilde f(s)=\beta s(1-s)\), and hence it is the unique global
solution of \eqref{Eq:Fisher-KPP-tree:nonlinear}.

Finally, suppose that \(u_0\) is radial with respect to the root \(o\).
Let \(\varphi\) be an automorphism of \(\TT_{q+1}\) fixing \(o\), and
define
\[
u^\varphi(t,x):=u(t,\varphi(x)).
\]
Since graph automorphisms commute with
\(\Delta_{\TT_{q+1}}\), the function \(u^\varphi\) satisfies the same
equation as \(u\). Moreover, by the radiality of \(u_0\),
\[
u^\varphi(0,x)
=
u_0(\varphi(x))
=
u_0(x).
\]
Uniqueness therefore implies
\[
u^\varphi=u.
\]
Since the stabilizer of \(o\) acts transitively on every geodesic sphere
centered at \(o\), \(u(t,\cdot)\) is constant on each such sphere.
Therefore,
\[
u(t,x)=u_n(t),
\qquad
n=d(o,x).
\]
\end{proof}

\begin{lemma}\label{lem:tilted-Green-new}
For \(t>0\) and \(n\in\ZZ\),
\begin{equation}\label{eq:tilted-Green-explicit-new}
\mathcal G_n(t)
=e^{-s_\ast t}
\left(\frac{a_\ast}{b_\ast}\right)^{n/2}
I_{|n|}(2\sqrt{a_\ast b_\ast}\,t).
\end{equation}
Moreover,
\begin{equation}\label{eq:tilted-moments-new}
\sum_{n\in\ZZ}\mathcal G_n(t)=1,
\qquad
\sum_{n\in\ZZ}n\mathcal G_n(t)=c_\ast t,
\qquad
\sum_{n\in\ZZ}(n-c_\ast t)^2\mathcal G_n(t)=s_\ast t.
\end{equation}
\end{lemma}

\begin{proof}
We first derive the explicit representation of the Green function by taking
the Laplace transform in time. For \(\lambda>0\), set
\[
\widehat{\mathcal G}_n(\lambda)
:=
\int_0^\infty e^{-\lambda t}\mathcal G_n(t)\,dt.
\]
Since
\[
(\mathscr L_\ast w)_n
=
a_\ast w_{n-1}
+
b_\ast w_{n+1}
-
s_\ast w_n,
\qquad
s_\ast=a_\ast+b_\ast,
\]
taking the Laplace transform in
\eqref{eq:critical-linear-whole-line-new} gives
\[
(\lambda+s_\ast)\widehat{\mathcal G}_n(\lambda)
-\delta_{0n}
=
a_\ast\widehat{\mathcal G}_{n-1}(\lambda)
+
b_\ast\widehat{\mathcal G}_{n+1}(\lambda).
\]
Hence, for \(n\neq0\),
\begin{equation}\label{eq:resolvent-recursion-tilted}
b_\ast\widehat{\mathcal G}_{n+1}(\lambda)
-
(\lambda+s_\ast)\widehat{\mathcal G}_n(\lambda)
+
a_\ast\widehat{\mathcal G}_{n-1}(\lambda)
=
0.
\end{equation}
The corresponding characteristic equation is
\[
b_\ast r^2-(\lambda+s_\ast)r+a_\ast=0,
\]
whose roots are
\[
r_\pm(\lambda)
=
\frac{
\lambda+s_\ast
\pm
\sqrt{(\lambda+s_\ast)^2-4a_\ast b_\ast}
}{
2b_\ast
}.
\]
Since
\[
0<r_-(\lambda)<1<r_+(\lambda),
\]
the decay of the resolvent as \(|n|\to\infty\) implies
\[
\widehat{\mathcal G}_n(\lambda)
=
C(\lambda)
\begin{cases}
r_-(\lambda)^n, & n\geq0,\\[1mm]
r_+(\lambda)^n, & n\leq0.
\end{cases}
\]
Using the resolvent equation at \(n=0\), together with the relations between
\(r_-\) and \(r_+\), we obtain
\[
C(\lambda)
=
\frac{1}{
\sqrt{(\lambda+s_\ast)^2-4a_\ast b_\ast}
}.
\]
Therefore,
\[
\widehat{\mathcal G}_n(\lambda)
=
\frac{1}{
\sqrt{(\lambda+s_\ast)^2-4a_\ast b_\ast}
}
\left(\frac{a_\ast}{b_\ast}\right)^{n/2}
\left(
\frac{
\lambda+s_\ast
-
\sqrt{(\lambda+s_\ast)^2-4a_\ast b_\ast}
}{
2\sqrt{a_\ast b_\ast}
}
\right)^{|n|}.
\]
We now use the classical Laplace transform identity
\[
\int_0^\infty e^{-\sigma t}I_m(\rho t)\,dt
=
\frac{1}{\sqrt{\sigma^2-\rho^2}}
\left(
\frac{\sigma-\sqrt{\sigma^2-\rho^2}}{\rho}
\right)^m,
\qquad \sigma>\rho,
\]
with
\[
\sigma=\lambda+s_\ast,
\qquad
\rho=2\sqrt{a_\ast b_\ast}.
\]
By uniqueness of the Laplace transform, it follows that
\[
\mathcal G_n(t)
=
e^{-s_\ast t}
\left(\frac{a_\ast}{b_\ast}\right)^{n/2}
I_{|n|}(2\sqrt{a_\ast b_\ast}\,t),
\]
which proves \eqref{eq:tilted-Green-explicit-new}.

We next obtain the moment identities from the generating function. Define
\[
F(t,z)
:=
\sum_{n\in\ZZ}\mathcal G_n(t)z^n,
\qquad z>0.
\]
Using the equation satisfied by \(\mathcal G_n\), we find
\[
\begin{aligned}
\partial_tF(t,z)
&=
a_\ast\sum_{n\in\ZZ}\mathcal G_{n-1}(t)z^n
+
b_\ast\sum_{n\in\ZZ}\mathcal G_{n+1}(t)z^n
-
s_\ast F(t,z)\\
&=
\left(
a_\ast z+b_\ast z^{-1}-s_\ast
\right)F(t,z).
\end{aligned}
\]
Since
\[
F(0,z)=1,
\]
we obtain
\begin{equation}\label{eq:tilted-generating-function}
F(t,z)
=
\exp\left\{
t\left(
a_\ast z+b_\ast z^{-1}-s_\ast
\right)
\right\}.
\end{equation}

Evaluating at \(z=1\) and using \(s_\ast=a_\ast+b_\ast\) gives
\[
F(t,1)=1,
\]
and therefore
\[
\sum_{n\in\ZZ}\mathcal G_n(t)=1.
\]

Introduce the operator
\[
D:=z\partial_z.
\]
Then
\[
DF(t,z)
=
\sum_{n\in\ZZ}n\mathcal G_n(t)z^n.
\]
From \eqref{eq:tilted-generating-function},
\[
DF(t,z)
=
t\left(a_\ast z-b_\ast z^{-1}\right)F(t,z).
\]
Hence, at \(z=1\),
\[
\sum_{n\in\ZZ}n\mathcal G_n(t)
=
(a_\ast-b_\ast)t.
\]
Since
\[
c_\ast=a_\ast-b_\ast,
\]
we obtain
\[
\sum_{n\in\ZZ}n\mathcal G_n(t)=c_\ast t.
\]

Similarly,
\[
D^2F(t,z)
=
\sum_{n\in\ZZ}n^2\mathcal G_n(t)z^n.
\]
A direct differentiation yields
\[
D^2F(t,z)
=
\left[
t\left(a_\ast z+b_\ast z^{-1}\right)
+
t^2\left(a_\ast z-b_\ast z^{-1}\right)^2
\right]F(t,z).
\]
Evaluating again at \(z=1\), we find
\[
\sum_{n\in\ZZ}n^2\mathcal G_n(t)
=
s_\ast t+c_\ast^2t^2.
\]
Consequently,
\[
\begin{aligned}
\sum_{n\in\ZZ}(n-c_\ast t)^2\mathcal G_n(t)
&=
\sum_{n\in\ZZ}n^2\mathcal G_n(t)
-2c_\ast t\sum_{n\in\ZZ}n\mathcal G_n(t)
+c_\ast^2t^2\sum_{n\in\ZZ}\mathcal G_n(t).\\
&=
s_\ast t.
\end{aligned}
\]
This proves \eqref{eq:tilted-moments-new}.
\end{proof}

\section{Moving Comparison and Leading-Edge Barriers}
\label{Appendix:Nonlinear:Barriers}

In this appendix we collect the comparison and barrier estimates used in
Proposition~\ref{prop:leading-edge-new}. The construction follows closely
\cite[Sections~5.1.1--5.1.2]{Besse-Faye-Roquejoffre-Zhang-2023-TAMS}.
We only indicate the modifications required by the present problem.

Throughout this appendix, we write
\[
y_n(t):=n-c_\ast t,
\qquad
T:=t+1,
\]
and recall that
\[
(\mathscr L_\ast z)_n
=
a_\ast(z_{n-1}-z_n)
+
b_\ast(z_{n+1}-z_n),
\]
where
\[
a_\ast-b_\ast=c_\ast,
\qquad
a_\ast+b_\ast=s_\ast.
\]
We also set
\[
A_\ast:=\sqrt{a_\ast b_\ast}=\alpha\sqrt q,
\qquad
\rho_\ast
:=
\frac12\log\frac{a_\ast}{b_\ast}
=
\lambda_\ast-\frac12\log q.
\]
Then
\[
a_\ast=A_\ast e^{\rho_\ast},
\qquad
b_\ast=A_\ast e^{-\rho_\ast},
\]
and therefore
\[
\mathscr L_\ast
=
A_\ast\mathscr L_{\rho_\ast},
\]
where
\[
(\mathscr L_\rho z)_n
:=
e^\rho(z_{n-1}-z_n)
+
e^{-\rho}(z_{n+1}-z_n).
\]

For barriers which may take negative values, we extend the logistic
nonlinearity by
\[
\widetilde f(s)
=
\begin{cases}
\beta s, & s\leq0,\\[1mm]
\beta s(1-s), & 0\leq s\leq1,\\[1mm]
0, & s\geq1.
\end{cases}
\]
Accordingly, we define
\begin{equation}\label{eq:Ndef-new}
\mathcal N_n(t;s)
:=
\beta s
-
e^{\lambda_\ast y_n(t)}
\widetilde f\left(
e^{-\lambda_\ast y_n(t)}s
\right).
\end{equation}
Then
\begin{equation}\label{eq:N-properties-new}
\mathcal N_n(t;s)\geq0
\quad\text{for }s\geq0,
\qquad
\mathcal N_n(t;s)=0
\quad\text{for }s\leq0,
\end{equation}
while
\begin{equation}\label{eq:N-quadratic-new}
\mathcal N_n(t;s)
=
\beta e^{-\lambda_\ast y_n(t)}s^2,
\end{equation}
whenever
\[
0\leq
e^{-\lambda_\ast y_n(t)}s
\leq1.
\]
Since \(0\leq u_n(t)\leq1\), the transformed solution satisfies
\begin{equation}\label{eq:v-extended-new}
\partial_t v_n
-
\mathscr L_\ast v_n
+
\mathcal N_n(t;v_n)
=
0,
\end{equation}
at every bulk index \(n\geq1\).

\subsection*{Moving comparison}

We first record the comparison principle used below.

\begin{lemma}[Comparison on a moving half-line]
\label{lem:moving-half-line-comparison-new}
Let \(t_0<T_1\), let
\(\zeta:[t_0,T_1]\to\RR\) be piecewise \(C^1\), and define
\[
\Omega_\zeta
:=
\left\{
(t,n):
t_0<t\leq T_1,\quad
n\in\ZZ,\quad
n\geq\zeta(t)
\right\}.
\]
Suppose that \(z\) and \(Z\) are bounded, continuous in time, and
piecewise \(C^1\), and satisfy
\[
\partial_tz_n
-
\mathscr L_\ast z_n
+
\mathcal N_n(t;z_n)
\leq
\partial_tZ_n
-
\mathscr L_\ast Z_n
+
\mathcal N_n(t;Z_n)
\]
at every interior point of \(\Omega_\zeta\). Assume moreover that
\[
z_n(t_0)\leq Z_n(t_0),
\]
on the initial section and that
\[
z_n(t)\leq Z_n(t),
\]
whenever
\[
\zeta(t)-1\leq n<\zeta(t).
\]
Then
\[
z_n(t)\leq Z_n(t)
\qquad
\text{for every }(t,n)\in\Omega_\zeta.
\]
\end{lemma}

\begin{proof}
Set
\[
d_n(t):=z_n(t)-Z_n(t).
\]
Since \(z\) and \(Z\) are bounded and
\(\mathcal N_n(t;\cdot)\) is locally Lipschitz, there exists \(L>0\)
such that, whenever \(d_n(t)>0\),
\[
-\Bigl(
\mathcal N_n(t;z_n(t))
-
\mathcal N_n(t;Z_n(t))
\Bigr)
\leq
Ld_n(t).
\]
Hence
\[
\partial_td_n-\mathscr L_\ast d_n
\leq
Ld_n,
\]
at every point where \(d_n>0\).

Define
\[
\widetilde d_n(t)
:=
e^{-(L+1)(t-t_0)}d_n(t).
\]
Then, whenever \(\widetilde d_n(t)>0\),
\[
\partial_t\widetilde d_n
-
\mathscr L_\ast\widetilde d_n
\leq
-\widetilde d_n.
\]

Since the moving half-line is unbounded, the spatial supremum of
\(\widetilde d\) need not be attained. Choose
\[
n_0<
\inf_{t\in[t_0,T_1]}\zeta(t)-1,
\]
and, for \(\varepsilon>0\), define
\[
D_n^\varepsilon(t)
:=
\widetilde d_n(t)
-
\varepsilon(n-n_0).
\]
Then
\[
D_n^\varepsilon(t)\to-\infty
\qquad
\text{as }n\to\infty,
\]
uniformly for \(t\in[t_0,T_1]\). Hence, if
\(D^\varepsilon\) becomes positive, its first positive maximum is attained
at a finite interior lattice site, say
\((t_\varepsilon,n_\varepsilon)\).

At this point,
\[
\partial_tD_{n_\varepsilon}^\varepsilon(t_\varepsilon)\geq0,
\]
and
\[
D_{n_\varepsilon\pm1}^\varepsilon(t_\varepsilon)
\leq
D_{n_\varepsilon}^\varepsilon(t_\varepsilon).
\]
Since \(a_\ast,b_\ast>0\),
\[
(\mathscr L_\ast D^\varepsilon)_{n_\varepsilon}(t_\varepsilon)
\leq0.
\]
On the other hand,
\[
\mathscr L_\ast(n-n_0)
=
b_\ast-a_\ast
=
-c_\ast,
\]
and therefore
\[
\begin{aligned}
\partial_tD_n^\varepsilon
-
\mathscr L_\ast D_n^\varepsilon
&=
\partial_t\widetilde d_n
-
\mathscr L_\ast\widetilde d_n
+
\varepsilon\mathscr L_\ast(n-n_0)
\\
&\leq
-\widetilde d_n-\varepsilon c_\ast
<0,
\end{aligned}
\]
whenever \(D_n^\varepsilon>0\). This contradicts
\[
\partial_tD_{n_\varepsilon}^\varepsilon
-
(\mathscr L_\ast D^\varepsilon)_{n_\varepsilon}
\geq0.
\]
Hence
\[
D_n^\varepsilon(t)\leq0,
\]
throughout \(\Omega_\zeta\). Letting \(\varepsilon\downarrow0\) gives
\[
\widetilde d_n(t)\leq0,
\]
and therefore
\[
z_n(t)\leq Z_n(t)
\qquad\text{in }\Omega_\zeta.
\]
\end{proof}

\subsection*{Auxiliary tail estimates}

The first estimate uses the geometry of the original tree.

\begin{lemma}[Fixed-time factorial tail]
\label{lem:fixed-time-tail-new}
Assume that
\[
\operatorname{supp}u_0
\subset
\{0,\ldots,R\},
\]
for some \(R\in\NN\). Then, for every fixed \(\tau>0\), there exists
\(C_\tau>0\) such that
\begin{equation}\label{eq:fixed-time-tail-new}
u_n(\tau)
\leq
C_\tau
\frac{
[\alpha(q+1)\tau]^{n-R}
}{
(n-R)!
},
\qquad
n>R.
\end{equation}
Consequently, \(v_n(\tau)\) decays faster than \(e^{-\sigma n}\) for
every fixed \(\sigma>0\).
\end{lemma}

\begin{proof}
Let \(U\) denote the solution of the linearized equation with initial
datum \(u_0\). Since
\[
\beta u(1-u)\leq \beta u
\qquad\text{for }0\leq u\leq1,
\]
the comparison principle gives
\[
0\leq u_n(t)\leq U_n(t).
\]

Let \((X_t)_{t\geq0}\) be the continuous-time simple random walk on
\(\TT_{q+1}\) generated by \(\alpha\Delta_{\TT_{q+1}}\). Since every
vertex has degree \(q+1\), the total jump rate is
\[
\alpha(q+1).
\]
Hence the number \(N_\tau\) of jumps performed by \(X_t\) during the
interval \([0,\tau]\) is a Poisson random variable with parameter
\[
\mu_\tau:=\alpha(q+1)\tau.
\]
The linear semigroup representation gives, for every vertex \(x\) with
\(d(o,x)=n\),
\[
U(\tau,x)
=
e^{\beta\tau}
\mathbb E_x\bigl[u_0(X_\tau)\bigr].
\]

Assume now that \(n>R\). Since
\[
\operatorname{supp}u_0
\subset
\{x\in\TT_{q+1}:d(o,x)\leq R\},
\]
a path starting from a vertex at distance \(n\) from the root cannot
reach the support of \(u_0\) in fewer than \(n-R\) jumps. Consequently,
\[
u_n(\tau)
\leq
U_n(\tau)
\leq
e^{\beta\tau}\|u_0\|_{\ell^\infty}
\mathbb P\{N_\tau\geq n-R\}.
\]

For \(m:=n-R\), using the Poisson distribution we obtain
\[
\begin{aligned}
\mathbb P\{N_\tau\geq m\}
&=
e^{-\mu_\tau}
\sum_{k=m}^{\infty}\frac{\mu_\tau^k}{k!}
\\
&=
e^{-\mu_\tau}
\frac{\mu_\tau^m}{m!}
\sum_{j=0}^{\infty}
\mu_\tau^j\frac{m!}{(m+j)!}
\\
&\leq
e^{-\mu_\tau}
\frac{\mu_\tau^m}{m!}
\sum_{j=0}^{\infty}\frac{\mu_\tau^j}{j!}
=
\frac{\mu_\tau^m}{m!}.
\end{aligned}
\]
Therefore
\[
u_n(\tau)
\leq
e^{\beta\tau}\|u_0\|_{\ell^\infty}
\frac{[\alpha(q+1)\tau]^{n-R}}{(n-R)!},
\]
which proves \eqref{eq:fixed-time-tail-new}.

Finally,
\[
v_n(\tau)
=
e^{\lambda_\ast(n-c_\ast\tau)}u_n(\tau),
\]
and hence, for every fixed \(\sigma>0\),
\[
e^{\sigma n}v_n(\tau)
\leq
C_\tau
e^{(\lambda_\ast+\sigma)n}
\frac{[\alpha(q+1)\tau]^{n-R}}{(n-R)!}.
\]
The right-hand side tends to zero as \(n\to\infty\), since the factorial
dominates every exponential. Thus \(v_n(\tau)\) decays faster than
\(e^{-\sigma n}\) for every fixed \(\sigma>0\).
\end{proof}

Let \(J\geq1\) be such that
\[
\operatorname{supp}w^0\subset[-J,J].
\]

\begin{lemma}[Far-field bound for the dipole]
\label{lem:far-field-dipole-new}
For every \(A>1\), there exists \(\eta_A>0\) such that
\begin{equation}\label{eq:far-field-dipole-new}
|w_n(t)|
\leq
\|w^0\|_{\ell^\infty}
\exp\left[
-A\left(
\frac{n-J-c_\ast t}{\sqrt{t+1}}
-\eta_A
\right)
\right],
\end{equation}
whenever
\[
n-J-c_\ast t
\geq
\eta_A\sqrt{t+1}.
\]
\end{lemma}

\begin{proof}
After the time rescaling
\[
\tau=A_\ast t,
\]
the equation generated by \(\mathscr L_\ast\) becomes
\[
\partial_\tau z_n
=
\mathscr L_{\rho_\ast}z_n.
\]
Thus, the exponential supersolution argument of
\cite[Section~5.1.1]{Besse-Faye-Roquejoffre-Zhang-2023-TAMS}
applies directly. Returning to the original time variable only changes
the positive constants, and gives \eqref{eq:far-field-dipole-new}.
\end{proof}

\subsection*{Leading-edge barriers}

Choose
\begin{equation}\label{eq:barrier-gamma-new}
\frac13<\gamma<\frac12,
\end{equation}
and
\begin{equation}\label{eq:barrier-parameters-new}
0<\delta<\nu,
\qquad
\nu<\gamma-\delta,
\qquad
2\nu<3\gamma-1.
\end{equation}
Define
\begin{equation}\label{eq:p-new}
p_n(t)
:=
T^{-3/2+\nu}
\cos\left(\frac{y_n(t)}{T^\gamma}\right).
\end{equation}
If
\[
\vartheta:=\frac{y_n(t)}{T^\gamma},
\]
then a direct computation gives
\begin{equation}\label{eq:p-residual-new}
\begin{aligned}
(\partial_t-\mathscr L_\ast)p_n
=
T^{-3/2+\nu}\Bigg(
&
\left[
\frac{\nu-3/2}{T}
+s_\ast(1-\cos T^{-\gamma})
\right]\cos\vartheta
\\
&+
\left[
c_\ast(T^{-\gamma}-\sin T^{-\gamma})
+\frac{\gamma\vartheta}{T}
\right]\sin\vartheta
\Bigg).
\end{aligned}
\end{equation}
Hence, on intervals on which \(\cos\vartheta\) remains uniformly
positive, the leading contribution is of order
\[
T^{-3/2+\nu-2\gamma}.
\]

At each endpoint of a positive cosine interval we use a width-one
truncation: the exterior value is set equal to zero and the endpoint
value is retained. At a left endpoint the deleted neighbor occurs with
coefficient \(a_\ast>0\), whereas at a right endpoint it occurs with
coefficient \(b_\ast>0\). Consequently, deleting a nonpositive exterior
value increases the residual for the upper barrier; deleting a
nonnegative exterior value decreases it for the lower barrier. Between
switching times this proves the required sign at the two endpoint sites.
At a switching time the one-sided limits have the same favorable sign,
so comparison is continued interval by interval. The inequalities
\(\nu<\gamma-\delta\) and \(2\nu<3\gamma-1\) ensure that the quadratic
and modulation errors are lower order than
\(T^{-3/2+\nu-2\gamma}\).

The barriers below use the piecewise truncations of \(p\) introduced in
\cite[Sections~5.1.1--5.1.2]
{Besse-Faye-Roquejoffre-Zhang-2023-TAMS}.
We denote by \(p^+\) the truncation used for the upper barrier and by
\(p^-\) the one used for the lower barrier. They agree with \(p\) in
their respective cosine regions and vanish outside them, up to the
width-one endpoint modifications required by the nearest-neighbor
operator. At the switching sites these modifications have the favorable
sign for the corresponding comparison inequality.

\begin{lemma}[Upper leading-edge barrier]
\label{lem:upper-barrier-new}
There exist \(T_+>0\), \(K_+>0\), a positive bounded increasing function
\(\xi_+\), and a nonnegative smooth function
\(\Xi:\RR\to[0,\infty)\) such that
\begin{equation}\label{eq:upper-barrier-def-new}
\overline v_n(t)
=
K_+\left[
\xi_+(t)w_n(t)
+
p_n^+(t)
+
\Xi\left(\frac{y_n(t)}{\sqrt T}\right)
\right],
\end{equation}
satisfies
\begin{equation}\label{eq:upper-barrier-res-new}
\partial_t\overline v_n
-
\mathscr L_\ast\overline v_n
+
\mathcal N_n(t;\overline v_n)
\geq0,
\end{equation}
for
\[
t\geq T_+,
\qquad
y_n(t)\geq-t^\delta.
\]
Moreover,
\begin{equation}\label{eq:upper-comparison-new}
v_n(t)\leq\overline v_n(t),
\end{equation}
in the same moving half-line.
\end{lemma}

\begin{proof}
We follow the construction in
\cite[Section~5.1.1]
{Besse-Faye-Roquejoffre-Zhang-2023-TAMS}
and only verify the points that differ in the present setting.

Set
\[
\sigma:=3\gamma-2\nu-1.
\]
By \eqref{eq:barrier-parameters-new}, \(\sigma>0\).
Choose any \(\kappa_+\in(0,1)\) and define
\[
\xi_+(t)
:=
1-\kappa_+T^{-\sigma}.
\]
Then
\[
\xi_+'(t)
=
\kappa_+\sigma T^{-\sigma-1}>0.
\]
After the fixed time rescaling
\[
\tau=A_\ast t,
\]
the linear equation is precisely the asymmetric equation considered in
the cited construction. Since this is only a constant rescaling of time,
the power relations in \eqref{eq:barrier-parameters-new} are unchanged.

We next define the far-field correction. Choose two constants
\[
1<a<A,
\]
where \(A\) is the exponent used in
Lemma~\ref{lem:far-field-dipole-new}. Choose numbers
\[
0<\eta_1<\eta_2,
\]
sufficiently large, and let
\[
\Gamma\in C^\infty(\RR;[0,1]),
\]
be nondecreasing and satisfy
\[
\Gamma(\eta)=0
\quad\text{for }\eta\leq\eta_1,
\qquad
\Gamma(\eta)=1
\quad\text{for }\eta\geq\eta_2.
\]
Define
\[
\Xi(\eta)
:=
2\|w^0\|_{\ell^\infty}
\Gamma(\eta)e^{-a(\eta-\eta_2)}.
\]
The far-field supersolution calculation is then the same as in
\cite[Section~5.1.1]
{Besse-Faye-Roquejoffre-Zhang-2023-TAMS},
after the rescaling \(\tau=A_\ast t\). Since \(A>a\),
Lemma~\ref{lem:far-field-dipole-new} ensures that the dipole term is
smaller than this correction in the far field.

It remains to verify two facts which are specific to the present
problem. First, in the moving domain,
\[
y_n(t)\geq-t^\delta,
\]
implies
\[
n\geq c_\ast t-t^\delta.
\]
Since \(c_\ast>0\) and \(\delta<1\), the whole moving domain lies in
\(n\geq1\) for all sufficiently large \(t\). Thus the exceptional
equation at the root does not enter the barrier argument.

Second, on the left lateral layer
\[
-t^\delta-1
\leq
y_n(t)
<
-t^\delta,
\]
we have
\[
v_n(t)
=
e^{\lambda_\ast y_n(t)}u_n(t)
\leq
e^{-\lambda_\ast t^\delta+O(1)},
\]
whereas \(p_n^+(t)\) has only algebraic decay. Hence the lateral
comparison holds for all sufficiently large \(t\).

Choose \(T_+\) so that all the preceding estimates hold. At the fixed
time \(T_+\), Lemma~\ref{lem:fixed-time-tail-new} gives factorial decay
of \(v_n(T_+)\), whereas the far-field correction has only exponential
decay. Thus, by increasing \(K_+\) if necessary, the initial comparison
also holds. Lemma~\ref{lem:moving-half-line-comparison-new} yields
\eqref{eq:upper-comparison-new}.

At the switching sites of \(p^+\), the truncation changes the barrier in
the favorable direction for a supersolution. Therefore the comparison
argument continues across these switching times.
\end{proof}

For the lower barrier, define
\[
\chi(t)
:=
\sup\left\{
r\geq1:
w_n(t)>0
\ \text{whenever}\
1\leq y_n(t)\leq r
\right\},
\]
and
\begin{equation}\label{eq:wtilde-new}
\widetilde w_n(t)
:=
w_n(t)
\mathbf 1_{\{1\leq y_n(t)\leq\chi(t)\}}.
\end{equation}
By Lemma~\ref{lem:critical-dipole-new}, for every fixed \(L>0\),
\[
\chi(t)\geq L\sqrt t,
\]
for all sufficiently large \(t\).

\begin{lemma}[Lower leading-edge barrier]
\label{lem:lower-barrier-new}
There exist \(T_->0\), \(\varepsilon_->0\), \(\kappa_->0\), and a
positive decreasing function
\[
\xi_-:[T_-,\infty)\to(0,\infty),
\]
bounded away from zero, such that
\begin{equation}\label{eq:lower-barrier-def-new}
\underline v_n(t)
=
\varepsilon_-
\left[
\xi_-(t)\widetilde w_n(t)
-
p_n^-(t)
\right],
\end{equation}
satisfies
\begin{equation}\label{eq:lower-barrier-res-new}
\partial_t\underline v_n
-
\mathscr L_\ast\underline v_n
+
\mathcal N_n(t;\underline v_n)
\leq0,
\end{equation}
for
\[
t\geq T_-,
\qquad
y_n(t)\geq t^\delta.
\]
Moreover,
\begin{equation}\label{eq:lower-comparison-new}
\kappa_-\underline v_n(t-1)
\leq
v_n(t),
\end{equation}
for all sufficiently large \(t\) in the same moving half-line.
\end{lemma}

\begin{proof}
We follow the construction in \cite[Section 5.1.2]{Besse-Faye-Roquejoffre-Zhang-2023-TAMS}
and indicate the modifications required in the present setting.

Choose \(T_->0\) sufficiently large and
\(\varepsilon_->0\) sufficiently small. Fix a sufficiently large
constant \(C>0\), and define \(\xi_-\) by
\[
\xi_-'(t)
=
-C\beta\varepsilon_-\xi_-^2(t)(t+1)^{-2},
\qquad
\xi_-(T_-)=1.
\]
Since
\[
\int_{T_-}^{\infty}(t+1)^{-2}\,dt<\infty,
\]
we have
\[
\frac{1}{\xi_-(t)}
=
1+
C\beta\varepsilon_-
\int_{T_-}^{t}(s+1)^{-2}\,ds,
\]
and therefore \(\xi_-\) is positive, decreasing, and bounded away
from zero on \([T_-,\infty)\).

After the time rescaling
\[
\tau=A_\ast t,
\qquad
A_\ast=\sqrt{a_\ast b_\ast},
\]
the linear part becomes precisely the operator considered in
\cite[Section 5.1.2]{Besse-Faye-Roquejoffre-Zhang-2023-TAMS}. The only difference
in the nonlinear term is the exponential factor appearing in
\eqref{eq:N-quadratic-new}. Since
\[
\rho_\ast
=
\lambda_\ast-\frac12\log q
<
\lambda_\ast,
\]
we have, whenever \(y_n(t)\ge t^\delta\),
\[
e^{-\lambda_*y_n(t)}
\le
e^{-\rho_*y_n(t)}.
\]
Thus the nonlinear contribution is bounded by the corresponding one
in the construction of Besse et al., up to a fixed positive
constant. Consequently, after choosing \(C\) sufficiently large and
\(\varepsilon_-\) sufficiently small, the function
\[
\underline v_n(t)
=
\varepsilon_-
\left(
\xi_-(t)\widetilde w_n(t)-p_n^-(t)
\right), 
\]
satisfies
\[
\partial_t\underline v_n
-
\mathscr L_\ast\underline v_n
+
\mathcal N_n(t;\underline v_n)
\leq 0.\]
for
\[
t\ge T_-,
\qquad
y_n(t)\ge t^\delta.
\]
We also choose \(\varepsilon_-\) small enough that, whenever
\(\underline v_n(t)>0\),
\[
0
\le
e^{-\lambda_*y_n(t)}\underline v_n(t)
\le1,
\]
so that the quadratic identity \eqref{eq:N-quadratic-new} is applicable.

The truncation in the definition of \(\widetilde w\) does not destroy
the subsolution inequality. Indeed, at the rightmost active site the
neighboring value of \(w\) which is removed is nonpositive. Since
\(a_*,b_*>0\), removing this value changes
\(\mathcal L_*\widetilde w\) in the favorable direction. The same
observation applies when the active endpoint changes with time.

It remains to compare the lower barrier with the true solution. We
first establish a uniform comparison between the dipole \(w\) and a
positive time translate of \(v\).

Since \(u_0\not\equiv0\) and \(u_0\ge0\), positivity of the heat
semigroup and the variation-of-constants formula imply
\[
u(t,x)>0,
\qquad
t>0,\quad x\in\mathbb T_{q+1}.
\]
Hence
\[
v_n(t)>0,
\qquad
t>0,\quad n\ge0.
\]

Set
\[
Z_n(s):=\varepsilon_0e^{-\beta s}w_n(s),
\qquad
0\le s\le T_-,
\]
where \(\varepsilon_0>0\) will be chosen below. Since \(w\) solves
the linear equation,
\[
\partial_s Z_n-\mathscr L_*Z_n+\beta Z_n=0.
\]
On the other hand, from the equation satisfied by \(v\),
\[
\partial_t v_n
-
\mathscr L_*v_n
=
-\beta e^{-\lambda_*y_n(t)}v_n^2,
\]
we obtain
\[
\partial_s v_n(s+1)
-
\mathcal L_*v_n(s+1)
+
\beta v_n(s+1)
=
\beta v_n(s+1)\bigl(1-u_n(s+1)\bigr)
\ge0.
\]

Recall that \(w^0\) is compactly supported and
\[
w_n^0\ge0,
\qquad n\ge1.
\]
Since \(v_n(1)>0\), we may choose \(\varepsilon_0>0\) sufficiently
small so that
\[
\varepsilon_0w_n^0\le v_n(1),
\qquad n\ge1.
\]
Indeed, this only has to be checked on the finite set
\(\operatorname{supp}w^0\cap\mathbb N\).

We also choose \(\varepsilon_0\) so that the lateral comparison at
\(n=0\) holds. Since
\[
\min_{0\le s\le T_-}v_0(s+1)>0, 
\]
and \(w_0\) is bounded on the compact interval \([0,T_-]\), after
decreasing \(\varepsilon_0\), if necessary, we have
\[
\varepsilon_0e^{-\beta s}w_0(s)
\le
v_0(s+1),
\qquad
0\le s\le T_-.
\]
The comparison principle on the fixed half-line \(n\ge1\) therefore
gives
\[
\varepsilon_0e^{-\beta s}w_n(s)
\le
v_n(s+1),
\qquad
0\le s\le T_-,
\quad n\ge1.
\]
In particular,
\begin{equation}\label{eq:dipole-true-comparison}
v_n(T_-+1)
\ge
c_0 w_n(T_-),
\qquad n\ge1,
\end{equation}
where
\[
c_0:=\varepsilon_0e^{-\beta T_-}>0.
\]

We now obtain the initial comparison for the lower barrier.
Outside the bounded support of \(p^-(T_-)\), whenever
\(\underline v_n(T_-)>0\), we necessarily have
\(\widetilde w_n(T_-)=w_n(T_-)>0\), and therefore
\[
\underline v_n(T_-)
\le
\varepsilon_- w_n(T_-),
\]
because \(0<\xi_-(T_-)=1\).
Hence, by \eqref{eq:dipole-true-comparison}, choosing
\(\kappa_->0\) sufficiently small so that
\[
\kappa_-\varepsilon_-\le c_0,
\]
gives
\[
\kappa_-\underline v_n(T_-)
\le
v_n(T_-+1),
\]
outside \(\operatorname{supp}p^-(T_-)\).
On the remaining finite set the same inequality follows, after
decreasing \(\kappa_-\) once more, from the strict positivity of
\(v_n(T_-+1)\). Thus
\begin{equation}\label{eq:initial-lower-comparison}
\kappa_-\underline v_n(T_-)
\le
v_n(T_-+1),
\end{equation}
on the whole initial section relevant to the moving comparison.

We may additionally require
\begin{equation}\label{eq:kappa-lower-choice}
0<\kappa_-\le e^{-\lambda_*c_*}.
\end{equation}
For \(t\ge T_-+1\), define
\[
Q_n(t):=\kappa_-\underline v_n(t-1).
\]
Notice that
\[
y_n(t-1)=y_n(t)+c_*.
\]
In particular, if \(y_n(t)\ge t^\delta\), then
\[
y_n(t-1)
=
y_n(t)+c_*
\ge
t^\delta
\ge
(t-1)^\delta,
\]
so that the subsolution inequality for \(\underline v\) is available
at \((t-1,n)\).

Suppose first that \(Q_n(t)>0\). Then
\(\underline v_n(t-1)>0\), and by the choice of
\(\varepsilon_-\),
\[
N_n(t-1;\underline v_n(t-1))
=
\beta
e^{-\lambda_*y_n(t-1)}
\underline v_n(t-1)^2.
\]
Moreover, by \eqref{eq:kappa-lower-choice},
\[
e^{-\lambda_*y_n(t)}Q_n(t)
=
\kappa_-e^{\lambda_*c_*}
e^{-\lambda_*y_n(t-1)}
\underline v_n(t-1)
\le1,
\]
and therefore
\[
N_n(t;Q_n(t))
=
\beta e^{-\lambda_*y_n(t)}Q_n(t)^2.
\]
Using the subsolution inequality for \(\underline v\), we obtain
\[
\begin{aligned}
\partial_tQ_n
-\mathcal L_*Q_n
+N_n(t;Q_n)
&\le
-\kappa_-
N_n(t-1;\underline v_n(t-1))
+
N_n(t;Q_n)
\\
&=
\beta\kappa_-
e^{-\lambda_*y_n(t-1)}
\underline v_n(t-1)^2
\left(
-1+\kappa_-e^{\lambda_*c_*}
\right)
\\
&\le0.
\end{aligned}
\]
If \(Q_n(t)\le0\), then \(N_n(t;Q_n)=0\). Since
\(\underline v_n(t-1)\le0\) in this case, also
\(N_n(t-1;\underline v_n(t-1))=0\), and the subsolution inequality
for \(\underline v\) gives
\[
\partial_tQ_n-\mathcal L_*Q_n\le0.
\]
Thus \(Q\) is a subsolution in the moving half-line
\[
y_n(t)\ge t^\delta.
\]

On its left lateral layer, the truncation \(p^-\) is chosen precisely
so that
\[
Q_n(t)\le0,
\]
whereas
\[
v_n(t)\ge0.
\]
Together with the initial comparison
\eqref{eq:initial-lower-comparison}, Lemma~\ref{lem:moving-half-line-comparison-new} therefore yields
\[
Q_n(t)\le v_n(t),
\]
for all sufficiently large \(t\) in the same moving half-line.
Equivalently,
\[
\kappa_-\underline v_n(t-1)\le v_n(t).
\]

\end{proof}

\begin{corollary}[Technical leading-edge estimate]
\label{cor:technical-leading-edge-new}
Choose
\begin{equation}\label{eq:eta-new}
\max\{\delta,\nu\}
<
\eta
<
\min\{\gamma,1/4\}.
\end{equation}
Then there exist \(T_\eta>0\) and constants \(0<c_0<C_0\) such that
\begin{equation}\label{eq:technical-leading-edge-new}
c_0t^\eta
\leq
t^{3/2}v_n(t)
\leq
C_0t^\eta, 
\end{equation}
whenever
\[
t\geq T_\eta,
\qquad
\frac12t^\eta
\leq
n-c_\ast t
\leq
2t^\eta.
\]
\end{corollary}

\begin{proof}
In this region,
Lemma~\ref{lem:critical-dipole-new} gives
\[
w_n(t)
\asymp
t^{-3/2+\eta}.
\]
Since
\[
y_n(t-1)
=
y_n(t)+c_\ast,
\]
the same estimate holds at time \(t-1\). Moreover,
\[
\frac{
T^{-3/2+\nu}
}{
t^{-3/2+\eta}
}
\longrightarrow 0,
\]
because \(\nu<\eta\). Finally, since \(\eta<1/2\), the argument of the
far-field correction satisfies
\[
\frac{y_n(t)}{\sqrt T}\longrightarrow0,
\]
and therefore \(\Xi(y_n(t)/\sqrt T)=0\) for all sufficiently large
\(t\).

The upper and lower barrier estimates consequently give
\[
v_n(t)
\asymp
t^{-3/2+\eta}, 
\]
uniformly in the stated region, which proves
\eqref{eq:technical-leading-edge-new}.
\end{proof}

\bibliographystyle{abbrv}
\bibliography{Biblio}

@article{Amos-1974,
 author = {Amos, D. E.},
 title = {Computation of modified {Bessel} functions and their ratios},
 fjournal = {Mathematics of Computation},
 journal = {Math. Comput.},
 issn = {0025-5718},
 volume = {28},
 pages = {239--251},
 year = {1974},
 language = {English},
 doi = {10.2307/2005830},
 zbMATH = {3436518},
 Zbl = {0277.65006}
}

@article{Aronson-Weinberger-1978-Advances,
 author = {Aronson, D. G. and Weinberger, H. F.},
 title = {Multidimensional nonlinear diffusion arising in population genetics},
 fjournal = {Advances in Mathematics},
 journal = {Adv. Math.},
 issn = {0001-8708},
 volume = {30},
 pages = {33--76},
 year = {1978},
 language = {English},
 doi = {10.1016/0001-8708(78)90130-5},
 zbMATH = {3634070},
 Zbl = {0407.92014}
}

@article{Bramson-1978,
 author = {Bramson, Maury D.},
 title = {Maximal displacement of branching {Brownian} motion},
 fjournal = {Communications on Pure and Applied Mathematics},
 journal = {Commun. Pure Appl. Math.},
 issn = {0010-3640},
 volume = {31},
 pages = {531--581},
 year = {1978},
 language = {English},
 doi = {10.1002/cpa.3160310502},
 zbMATH = {3562205},
 Zbl = {0361.60052}
}

@article{Besse-Faye-2021-BML,
 author = {Besse, Christophe and Faye, Gr{\'e}gory},
 title = {Spreading properties for {SIR} models on homogeneous trees},
 fjournal = {Bulletin of Mathematical Biology},
 journal = {Bull. Math. Biol.},
 issn = {0092-8240},
 volume = {83},
 number = {11},
 pages = {27},
 note = {Id/No 114},
 year = {2021},
 language = {English},
 doi = {10.1007/s11538-021-00948-7},
 zbMATH = {7432913},
 Zbl = {1475.92148}
}

@article{Besse-Faye-Roquejoffre-Zhang-2023-TAMS,
 author = {Besse, Christophe and Faye, Gr{\'e}gory and Roquejoffre, Jean-Michel and Zhang, Mingmin},
 title = {The logarithmic {Bramson} correction for {Fisher}-{KPP} equations on the lattice {{\(\mathbb{Z} \)}}},
 fjournal = {Transactions of the American Mathematical Society},
 journal = {Trans. Am. Math. Soc.},
 issn = {0002-9947},
 volume = {376},
 number = {12},
 pages = {8553--8619},
 year = {2023},
 language = {English},
 doi = {10.1090/tran/9007},
 zbMATH = {7772568},
 Zbl = {1534.34024}
}

@article{Cowling-Meda-Setti-2000-TAMS,
 author = {Cowling, Michael and Meda, Stefano and Setti, Alberto G.},
 title = {Estimates for functions of the {Laplace} operator on homogeneous trees},
 fjournal = {Transactions of the American Mathematical Society},
 journal = {Trans. Am. Math. Soc.},
 issn = {0002-9947},
 volume = {352},
 number = {9},
 pages = {4271--4293},
 year = {2000},
 language = {English},
 doi = {10.1090/S0002-9947-00-02460-0},
 zbMATH = {1460657},
 Zbl = {0949.43008}
}

@article{Fang-Li-Lou-Wang-2026-JDE,
 author = {Fang, Jian and Li, Yifei and Lou, Yijun and Wang, Jian},
 title = {Fisher-{KPP} waves and the minimal speed on hexagonal lattice},
 fjournal = {Journal of Differential Equations},
 journal = {J. Differ. Equations},
 issn = {0022-0396},
 volume = {467},
 pages = {36},
 note = {Id/No 114263},
 year = {2026},
 language = {English},
 doi = {10.1016/j.jde.2026.114263},
 zbMATH = {8191847}
}

@book{Gradshteyn-Ryzhik-2015,
 author = {Gradshteyn, I. S. and Ryzhik, I. M.},
 title = {Table of integrals, series, and products. {Translated} from the {Russian}. {Translation} edited and with a preface by {Victor} {Moll} and {Daniel} {Zwillinger}},
 edition = {8th updated and revised ed.},
 isbn = {978-0-12-384933-5; 978-0-12-384934-2},
 year = {2015},
 publisher = {Amsterdam: Elsevier/Academic Press},
 language = {English},
 url = {www.sciencedirect.com/science/book/9780123849335},
 zbMATH = {6369300},
 Zbl = {1300.65001}
}

@article{Guo-Wu-2008-Osaka,
 author = {Guo, Jong-Shenq and Wu, Chang-Hong},
 title = {Existence and uniqueness of traveling waves for a monostable 2-{D} lattice dynamical system},
 fjournal = {Osaka Journal of Mathematics},
 journal = {Osaka J. Math.},
 issn = {0030-6126},
 volume = {45},
 number = {2},
 pages = {327--346},
 year = {2008},
 language = {English},
 zbMATH = {5309630},
 Zbl = {1155.34016}
}

@article{Hoffman-Holzer-2019,
 author = {Hoffman, Aaron and Holzer, Matt},
 title = {Invasion fronts on graphs: the {Fisher}-{KPP} equation on homogeneous trees and {Erd{\H{o}}s}-{R{\'e}yni} graphs},
 fjournal = {Discrete and Continuous Dynamical Systems. Series B},
 journal = {Discrete Contin. Dyn. Syst., Ser. B},
 issn = {1531-3492},
 volume = {24},
 number = {2},
 pages = {671--694},
 year = {2019},
 language = {English},
 doi = {10.3934/dcdsb.2018202},
 zbMATH = {7000388},
 Zbl = {1410.35267}
}

@article{Hupkes-Jukic-Stehlik-Sigler-2023,
 author = {Hupkes, Hermen Jan and Juki{\'c}, Mia and Stehl{\'{\i}}k, Petr and {\v{S}}v{\'{\i}}gler, Vladim{\'{\i}}r},
 title = {Propagation reversal for bistable differential equations on trees},
 fjournal = {SIAM Journal on Applied Dynamical Systems},
 journal = {SIAM J. Appl. Dyn. Syst.},
 issn = {1536-0040},
 volume = {22},
 number = {3},
 pages = {1906--1944},
 year = {2023},
 language = {English},
 doi = {10.1137/22M1502203},
 zbMATH = {7729131},
 Zbl = {1523.34015}
}

@article{Zinner-Harris-Hudson-1993-JDE,
 author = {Zinner, Bertram and Harris, Glenn and Hudson, William},
 title = {Traveling wavefronts for the discrete {Fisher}'s equation},
 journal = {J. Differential Equations},
 volume = {105},
 number = {1},
 pages = {46--62},
 year = {1993},
 doi = {10.1006/jdeq.1993.1082}
}

@article{Chen-Guo-2003-MathAnn,
 author = {Chen, Xinfu and Guo, Jong-Shenq},
 title = {Uniqueness and existence of traveling waves for discrete quasilinear monostable dynamics},
 journal = {Math. Ann.},
 volume = {326},
 number = {1},
 pages = {123--146},
 year = {2003},
 doi = {10.1007/s00208-003-0414-0}
}

@article{AlHaj-Monneau-2024-JDE,
 author = {Al Haj, M. and Monneau, R.},
 title = {Traveling waves for discrete reaction-diffusion equations in the general monostable case},
 fjournal = {Journal of Differential Equations},
 journal = {J. Differ. Equations},
 issn = {0022-0396},
 volume = {378},
 pages = {707--756},
 year = {2024},
 language = {English},
 doi = {10.1016/j.jde.2023.10.017},
 zbMATH = {7765637},
 Zbl = {1527.35118}
}

@article{Carr-Chmaj-2004-PAMS,
 author = {Carr, Jack and Chmaj, Adam},
 title = {Uniqueness of travelling waves for nonlocal monostable equations},
 fjournal = {Proceedings of the American Mathematical Society},
 journal = {Proc. Am. Math. Soc.},
 issn = {0002-9939},
 volume = {132},
 number = {8},
 pages = {2433--2439},
 year = {2004},
 language = {English},
 doi = {10.1090/S0002-9939-04-07432-5},
 zbMATH = {2091051},
 Zbl = {1061.45003}
}

@article{Bateman-1943-BAMS,
 author = {Bateman, H.},
 title = {Some simple differential difference equations and the related functions},
 fjournal = {Bulletin of the American Mathematical Society},
 journal = {Bull. Am. Math. Soc.},
 issn = {0002-9904},
 volume = {49},
 pages = {494--512},
 year = {1943},
 language = {English},
 doi = {10.1090/S0002-9904-1943-07927-X},
 zbMATH = {3100626},
 Zbl = {0061.20201}
}

@article{Bouin-Henderson-Ryzhik-2020-Poicare,
 author = {Bouin, Emeric and Henderson, Christopher and Ryzhik, Lenya},
 title = {The {Bramson} delay in the non-local {Fisher}-{KPP} equation},
 fjournal = {Annales de l'Institut Henri Poincar{\'e}. Analyse Non Lin{\'e}aire},
 journal = {Ann. Inst. Henri Poincar{\'e}, Anal. Non Lin{\'e}aire},
 issn = {0294-1449},
 volume = {37},
 number = {1},
 pages = {51--77},
 year = {2020},
 language = {English},
 doi = {10.1016/j.anihpc.2019.07.001},
 zbMATH = {7152414},
 Zbl = {1436.35239}
}

@article{Alfaro-Giletti-Xiao-2025-MathAnn,
 author = {Alfaro, Matthieu and Giletti, Thomas and Xiao, Dongyuan},
 title = {The {Bramson} correction in the {Fisher}-{KPP} equation: from delay to advance},
 fjournal = {Mathematische Annalen},
 journal = {Math. Ann.},
 issn = {0025-5831},
 volume = {392},
 number = {4},
 pages = {5275--5316},
 year = {2025},
 language = {English},
 doi = {10.1007/s00208-025-03231-5},
 zbMATH = {8095765},
 Zbl = {1573.35354}
}

@book{Deimling-1977-Springer,
 author = {Deimling, Klaus},
 title = {Ordinary differential equations in {Banach} spaces},
 fseries = {Lecture Notes in Mathematics},
 series = {Lect. Notes Math.},
 issn = {0075-8434},
 volume = {596},
 year = {1977},
 publisher = {Springer, Cham},
 language = {English},
 doi = {10.1007/bfb0091636},
 zbMATH = {3561733},
 Zbl = {0361.34050}
}

@book{Smith-1995-Book,
 author = {Smith, Hal L.},
 title = {Monotone dynamical systems: an introduction to the theory of competitive and cooperative systems},
 fseries = {Mathematical Surveys and Monographs},
 series = {Math. Surv. Monogr.},
 issn = {0076-5376},
 volume = {41},
 isbn = {0-8218-0393-X},
 year = {1995},
 publisher = {Providence, RI: American Mathematical Society},
 language = {English},
 zbMATH = {758785},
 Zbl = {0821.34003}
}

\end{document}